%% file: LogSurfacesAut.tex
\documentclass[11pt]{amsart}

\usepackage{amssymb}
\usepackage[cmtip,all]{xy}

\newtheorem{thm}{Theorem}
\newtheorem{prop}[thm]{Proposition}
\newtheorem{lem}[thm]{Lemma}
\newtheorem{cor}[thm]{Corollary}

\theoremstyle{definition}

\theoremstyle{remark}
\newtheorem{rem}[thm]{Remark}

\numberwithin{equation}{section}

\newcommand{\findem}{\hspace*{\fill}$\Box$ \smallskip}
\newcommand{\Z}{\mathbb{Z}}

\newcommand{\C}{\mathbb{C}}

\newcommand{\Q}{\mathbb{Q}}
\newcommand{\F}{\mathbb{F}}
\newcommand{\CP}{\mathbb{P}}
\newcommand{\cpo}{\mathbb{P}^1}
\newcommand{\cpd}{\mathbb{P}^2}
\newcommand{\cpt}{\mathbb{P}^3}

\newcommand{\reste}{\mbox{\sffamily R}}
\newcommand{\aux}{{\rm Aux\,}}
\newcommand{\ind}{{\rm ind}}

\input{graph_surface}

\begin{document}

\title{Variations on log Sarkisov program for surfaces}

\author{Adrien Dubouloz}
\address{Institut de Math\'ematiques de Bourgogne, Universit\'e de Bourgogne, 9 avenue Alain Savary - BP 47870, 21078 Dijon cedex, France}
\email{adrien.dubouloz@u-bourgogne.fr}

\author{St\'ephane Lamy}
\address{Institut Camille Jordan, Universit\'e Lyon 1, 43 bd du 11 nov. 1918, 69622 Villeurbanne cedex, France }
\email{lamy@math.univ-lyon1.fr}

\date{January 2009}

\begin{abstract}
 Let $(S,B_S)$ be the log pair associated with a compactification of a given smooth quasi-projective surface $V$. Under the assumption that the boundary $B_S$ is irreducible, we propose an algorithm, in the spirit of the (log) Sarkisov program, to factorize any automorphism of $V$ into a sequence of elementary links in the framework of the log Mori theory. The new noteworthy feature of our algorithm is that all the blow-ups and contractions involved in the process occur on the boundary. Moreover, when the compactification $S$ is smooth we obtain a description of the automorphism group of $V$ which is reminiscent of a presentation by generators and relations.\\
\end{abstract}
%

\maketitle

\section*{Introduction}

Let $V$ be a smooth quasi-projective surface. We plan to describe the automorphisms of $V$ when there exists a compactification $V \subset S$ where $S$ is a (possibly singular) projective surface with $S \setminus V$ equal to an irreducible curve. More precisely, we look for a decomposition in the framework of  log Mori theory for automorphisms of $V$ that do not extend as biregular automorphisms on $S$. In this introduction  we suppose the reader has some familiarity with the basics of Mori theory (\cite{Mat} is an agreeable introductory book).\\

Recall that a $\Q$-factorial normal variety with at most terminal singularities $X$ is called a Mori fiber space if there exists a morphism $g: X \to Y$ with the following properties: $g$ has connected fibers, $Y$ is a normal variety of a dimension strictly less than $X$, and all the curves contracted  by $g$ are numerically proportional and of negative intersection with the canonical divisor $K_X$. A Mori fiber space should be thought of as a ``simplest possible'' representative in its birational class.

The Sarkisov program, written out in dimension 3 by Corti in 1995 \cite{Cor}, is an algorithm to  decompose a birational map $f: Y \dashrightarrow Y'$ between Mori fiber spaces into so-called elementary links. The algorithm works in principle in arbitrary dimension: Hacon and McKernan have recently announced a proof as a consequence of the results in \cite{BCHM}. The general idea is that one decomposes $f$ with the help of a sequence of intermediate varieties between  $Y$ and $Y'$, and that we have control of the complexity of these varieties in the sense that,  modulo isomorphism in codimension 1, at most one divisorial contraction is sufficient to come back to a Mori fiber space. 

Here is a brief description of the algorithm, where we try to avoid to enter into too much technicality. We start by taking a resolution $Y\stackrel{\pi}{\leftarrow} X \stackrel{\pi'}{\rightarrow} Y'$ of the base points of  $f$, where $X$ is a smooth  projective variety, and we choose an ample divisor $H'$ on $Y'$. We note $H_Y \subset Y$ (or $H_X \subset X$, etc...) the strict transform of a general member of the linear system $|H'|$, and  $C_i \subset X$ the irreducible components of the exceptional locus of $\pi$. We write down the ramification formulas  
$$ K_X = \pi^*K_Y + \sum c_i C_i \qquad \mbox{ and } \qquad 
 H_X = \pi^*H_Y - \sum m_i C_i $$
and we define the maximal multiplicity $\lambda$ as the  maximum of the $\lambda_i = \frac{m_i}{c_i}$ (technically $1/\lambda$ is defined as a canonical threshold). On the other hand we define the degree $\mu$ of $f$ as the rational number $\frac{H_Y.C}{-K_Y.C}$ where $C$ is any curve contained in a fiber of the Mori fibration on $Y$. In the case $\lambda > \mu$, that we feel is the general case, the algorithm predicts the  existence of a maximal extraction  $Z \to Y$ (we take the terminology from \cite{Cor2}, in \cite{Mat} the same operation is called a maximal divisorial blow-up), which by definition is an extremal divisorial contraction whose exceptional divisor realizes the  maximal multiplicity $\lambda$. Then either  $Z$ is itself a Mori fiber space, or there exists another  extremal divisorial contraction on  $Z$ (possibly preceded  by a sequence of log flips, that are isomorphisms in  codimension 1) that brings us back to a Mori  fiber space. These operations done, one shows that we have simplified  $f$ in the sense that: either $\mu$ went down; or $\mu$ remained constant but $\lambda$ went down; or $\mu$ and $\lambda$ remained constant but the number of exceptional divisors in $X$ realizing the  multiplicity $\lambda$ went down. As we can see the algorithm is quite complex, not to mention the case  $\lambda \le \mu$ and the proof of the termination of the algorithm, which are also intricate and that we do not detail further.

In 1997 Bruno and Matsuki \cite{BM} published a logarithmic version of this algorithm: the log Sarkisov program. In this new situation there exist some  distinguished divisors $B_Y$ and $B_{Y'}$ on the varieties $Y$ and $Y'$: this arises naturally when  $Y$ and $Y'$ are  compactifications of a fixed quasi-projective variety $V$; by analogy with this case we say in general that $B_Y$ is the boundary divisor of $Y$. The idea is that the algorithm remains formally the same, where $K_Y+B_Y$ now plays the role of the canonical divisor $K_Y$. The degree $\mu$ in this context is defined as $\mu = \frac{H_Y.C}{-(K_Y+B_Y).C}$, where $C$ is in some fiber of the log Mori fibration on  $Y$. In addition to the ramification formulas  for $K_X$ and $H_X$ we now have a similar formula for the boundary:
$$B_X =    \pi^*B_Y - \sum b_i C_i$$
and the maximal multiplicity is defined as the maximum of the $\lambda_i = \frac{m_i}{c_i-b_i}$. Bruno and  Matsuki worked out a  log Sarkisov algorithm in two  cases:
\begin{enumerate}
\item In dimension  3, for boundary divisors whose all coefficients are strictly less than 1 (the precise technical condition is $(Y,B_Y)$ klt, for kawamata log terminal);
\item In dimension 2, for boundary divisors whose coefficients are less or equal to 1 (the technical condition is dlt, for divisorially log terminal). 
\end{enumerate}

The expressed hope is that a  refinement of such an algorithm could allow us to understand the structure of polynomial automorphisms of $\C^3$. We have in mind to  compactify $\C^3$ by the projective space $\cpt$, and to apply the algorithm to the birational map from $\cpt$ to $\cpt$ induced by an automorphism of $\C^3$. A technical problem is that the boundary in this situation is the plane at infinity, with  coefficient +1, and therefore we are  not in the  klt framework. Nevertheless in  dimension 2 this obstacle  disappears, and we might feel free to think that everything is done in the case of surfaces.
  
Now here is the example that initially gave us the motivation to write on the log Sarkisov program in dimension 2 in spite of the existence of the results by Bruno-Matsuki. Let us consider an affine quadric surface  $V$, for instance we can take $V = \{w^2 + uv =1\} \subset \C^3$. Such a surface is isomorphic to $\cpo\times \cpo$ minus a diagonal $D$. Let $f$ be the rational map 
$$f: (x,y) \in \C^2 \dashrightarrow \left( x+ \frac{1}{x+y},y-\frac{1}{x+y}\right) \in \C^2.$$
This map preserves the levels $x+y = cte$, extends as a birational map from $S = \cpo \times \cpo$ to $S' = \cpo \times \cpo$, and induces  an isomorphism on  $V =   \cpo \times \cpo \setminus D$ where $D$ is the diagonal $x+y = 0$. The unique (proper) base point is the  point $p = [1:0],[1:0]$, and the unique contracted curve is the diagonal $D$. We can resolve  $f$ by performing 4 blow-ups that give rise to divisors  $C_1, \cdots, C_4$ arranged as on figure \ref{fig:1} (all these claims are not difficult to check by straightforward calculations in local charts; the reader may also look in  \cite{Lamyquad}).

The divisor $C_0$ is (the strict transform of) the diagonal on $S$, and $C_4$ is the diagonal on $S'$. Let us choose $H'=D$ as an ample divisor on $S'$, then we compute the coefficients in the ramification formulas:
$$  K_X = \pi^*K_S + \sum c_i C_i, \,
    B_X = \pi^*B_S - \sum b_i C_i \mbox{ and } 
    H_X = \pi^*H_S - \sum m_i C_i, $$
in order to deduce the $\lambda_i = \frac{m_i}{c_i-b_i}$. The $c_i$ and $b_i$ are easy to compute; for the $m_i$ it is sufficient to check that in this particular example the strict transform $H_S$ of a general member of $|D|$ is a smooth curve. The results are tabulated in the figure \ref{fig:1}.

 \begin{figure}[ht]
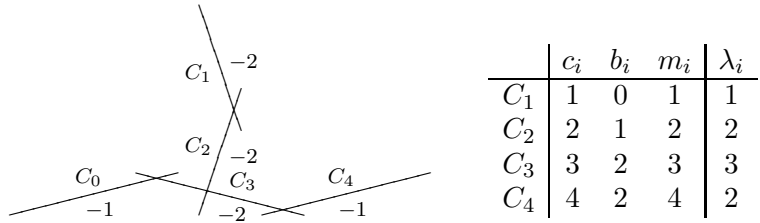

 $$\dessinresolutionbasic \qquad 
\begin{array}{c|ccc|c}
     & c_i & b_i & m_i & \lambda_i \\ \hline
 C_1 & 1   & 0   & 1   & 1 \\
 C_2 & 2   & 1   & 2   & 2 \\
 C_3 & 3   & 2   & 3   & 3 \\
 C_4 & 4   & 2   & 4   & 2 \\ 
 \end{array} $$
\caption{Resolution of $f$ and coefficients in the ramification formulas.} \label{fig:1}
\end{figure}

Thus the maximal multiplicity is realized by the divisor $C_3$. We can construct the maximal extraction of the  maximal singularity in the following way: blow-up three times to produce  $C_1,C_2$ and $C_3$, then contract $C_1$ and $C_2$ creating a singular point (this is a so-called Hirzebruch-Jung singularity, noted $A_{3,2}$). We obtain a surface $Z$ that compactifies the affine quadric $V$ by two curves: $C_0$ and $C_3$ (the latter supporting the unique singular point on the surface). After this maximal extraction is made we notice that there exists  4 curves on  $Z$ that correspond to  $K+B$ extremal rays:
 \begin{itemize}
  \item The strict transforms of the 2 rules $D_{+}$ and $D_{-}$ crossing at $p$: it is one of these two curves that the Bruno-Matsuki's algorithm imposes to contract  (precisely: the one that was a fiber for the chosen structure of Mori fiber space on $\cpo \times \cpo$);
  \item $C_3$, which is the exceptional divisor associated with the maximal multiplicity (that we have just constructed);
  \item $C_0$, which is the strict transform of the diagonal on $S$: it is this curve that our algorithm will impose to contract.
 \end{itemize}
This elementary example shows that the log Sarkisov algorithm   proposed by Bruno and Matsuki is not fully satisfying in the  sense that there is no reason why it should  respect the surface $V$ (the two authors were well aware of this fact, see \cite[problem 4.4]{BM}). It would be  natural to hope for an algorithm where all the blow-ups and contractions occur on the boundary divisor. This is such an algorithm, ``a variation on the log Sarkisov theme'', that we propose in this paper.
 
Our main result reads as follows, where the notion of ``admissible compactification'' will be  defined and discussed on paragraph \ref{section:cadre} below. Note that the surfaces $S_i$ are not supposed to be Mori fiber spaces.\\

\begin{thm}\label{th:main}
Let $f:V\stackrel{\sim}{\rightarrow}V'$ be an isomorphism of  smooth quasi-projective surfaces,  and let $S,S'$ be  admissible compactifications of $V$ (or equivalently of $V'$) such that the boundary divisors  $B_S, B_{S'}$ have irreducible support. Then if the induced birational map  $f: S \dashrightarrow S'$ is not an isomorphism, we can decompose $f$ into a finite sequence of $n$ links of the following form
 $$\xymatrix{
 &Z_i \ar[dl] \ar[dr] \\
 S_{i-1} && S_i
 }$$
where  $S_0 = S,S_1, \cdots,S_n = S'$ are admissible compactifications of  $V$ with an irreducible boundary,  $Z_i$ is for all $i = 1,\cdots,n$ an admissible compactification of $V$ with two boundary components,  and $Z_i \to S_{i-1}$, $Z_i \to S_i$ are the divisorial contractions associated with each one of the two $K+B$ extremal rays with support in the boundary of ${Z_i}$. 
\end{thm}

Our motivation for proving such a result is twofold. First it seems to us that the shift in focus from the existence of a Mori fiber structure to the property of admitting a compactification by one irreducible divisor may lead to the right statement for studying automorphisms of $\C^3$; of course this remains very hypothetical. 

More concretely, this result is also a first step to understand precisely the automorphism group of a non compact surface: in the case of a surface $V$ which admits a smooth compactification with an irreducible boundary, we are able to obtain a kind of presentation by generators and relations for the automorphism group of $V$ (see proposition \ref{prop:main}). However the striking fact is that in general the ``generators'' can not be canonically interpreted as automorphisms of $V$ but only as birational transformations $S \dashrightarrow S'$ between possibly non isomorphic smooth compactifications of $V$, inducing isomorphisms between one copy of $V$ to another. This remark was already apparent in the work by Danilov and Gizatullin \cite{GD1, GD2}; however our approach is somewhat different (for instance compare our proposition \ref{prop:main} and the example in section \ref{expl:+4} with the presentation given in \cite[$\S$7]{GD2})

In the case of surfaces admitting only singular compactification a new phenomenon arise, that is called reversion. We give an example in the paragraph  \ref{par:chain}; a systematic study of this situation can be found in \cite{BD}.  

\section{The factorization algorithm}

\subsection{Admissible surfaces} \label{section:cadre}
Here we discuss the class of admissible compactifications, and show that the  hypothesis made on the singularities and the geometry of the boundary are, in a sense,  optimal. Let us mention that it is relevant  to consider  quasi-projective surfaces $V$ and not only affine ones; for instance $\cpo \times \cpo$ minus a fiber is a non affine surface with a rich group of automorphisms.\\
 
\subsubsection*{Singularities.} 
First of all an automorphism of a quasi-projective normal surface $V$ extends as an automorphism of the minimal desingularization of $V$; this remark allows us to restrict without loss of generality to the case of a smooth surface  $V$. On the other hand it is natural to allow some kind of  singularities on the compactifications $S$  of $V$, indeed the log MMP can produce a singular variety after an extremal contraction even if the variety we started with was smooth (this is true even for surfaces).   
The widest framework where the Mori Program is (essentially) established in arbitrary dimension is the one of pairs $(Y,B_Y)$ with  klt (kawamata log terminal) singularities. However, if we want to work with a reduced boundary we have to allow a larger class of singularities; the most reasonnable choice seems then to work with dlt (divisorially log terminal) singularities. For the general definition of dlt singularities we refer the reader to   \cite[def. 2.37]{KM}; the interested reader may consult chapter 3 of \cite{Cflips} by Fujino for a complete discussion on the diverse existing definitions of  log terminal singularities. In the case  of a pair $(S,B_S)$ with  $S$ a projective surface, $B_S = \sum E_i$ a non empty reduced divisor (i.e. all the coefficients of the $E_i$ are equal to $1$) and $S \setminus B_S$ smooth, the dlt condition is equivalent to the following properties: 
\begin{itemize}
 \item Any singular point $p$ of $S$ is a point of $B_S$ which is not a crossing $E_i \cap E_j$;
 \item The $E_i$ are smooth irreducible curves with normal crossings;
 \item A singular point $p$ is locally isomorphic to a quotient of $\C^2$ by a cyclic group, that is to say the Hirzebruch-Jung singularities $A_{n,q}$ are the only ones allowed (the reader may find a discussion of these classical singularities for instance in \cite[p.99]{BHPV}). Furthermore if $C_1,\cdots,C_s$ is the minimal chain of rational curves (each with self-intersection $\le -2$) that desingularizes $p$, this is the first curve  $C_1$ that meets transversaly the strict transform of $B_S$.  
\end{itemize}
With the same notation, in the ramification formula 
$$K_{\overline{S}}+ B_{\overline{S}} = \pi^*(K_S+B_S) + \sum a_i C_i$$
we have $a_i >0$ for all $i$ (this is in fact the definition of $p$ being a log terminal singularity).  
The characterization above comes from  \cite[prop. 2.42]{KM} and from the local description of surfaces log terminal  singularities  when the boundary is reduced and non empty, that can be found  in \cite[see in particular p.57, case(3)]{Kflips}. \\
 
\subsubsection*{Geometry of the boundary.}
A first observation is that it is unreasonable to try  to extend the statement of the theorem to the case of surfaces with reducible boundary. Let us suppose indeed   that $f:S\dashrightarrow S'$ is a  birational map with $B_S$ reducible. Let $g:Z\rightarrow S$ be a birational morphism whose exceptional locus $E$ is  irreducible. Then either the image  $g(E)$  is located at the intersection  point of two components of $B_S$, and the remark \ref{rem:2voisins} (see further, at the end of  paragraph \ref{section:preuve}) implies that $g$ can not be a  $K+B$ extremal contraction ; either $g(E)$ is a general point of a component $E_i$ of $B_S$, and this time the same argument forbids  the contraction of $E_i$ to be $K+B$ extremal. A very simple explicit example is given by the identity map  $\C^2 \to \C^2$ viewed as a map from $\cpo \times \cpo$ to $\cpd$: it admits an unique base point located at the  intersection of the two rules at infinity, and the blow-up of this point is not a $K+B$ extremal contraction (it is only a $K$ extremal contraction).

A second  observation is that the  existence of a birational map that is not a morphism $f: S \dashrightarrow S'$ imposes strong  constraints on the irreducible boundary $E_0 = B_S$. Let us introduce some  notations that will also serve in the proof of the theorem. Let  $\pi:\tilde{S}\rightarrow S$ and  $\pi':\tilde{S}' \rightarrow S'$ be minimal resolutions of the singularities of $S$ and $S'$ respectively; from the characterization of dlt singularities we deduce that the total  transforms of $B_S$  and $B_{S'}$ are simple normal crossing divisors of $\tilde{S}$ and $\tilde{S'}$ respectively. Let $\tilde{S}\stackrel{\sigma}{\leftarrow} X \stackrel{\sigma'}{\rightarrow} \tilde{S}'$ be  a minimal resolution of the  base points of $f$ view as a birational map between  $\tilde{S}$ and $\tilde{S}'$.
\[\xymatrix@C4pc{& X \ar[dl]_{\sigma}  \ar[dr]^{\sigma'} \\  \tilde{S} \ar[d]_{\pi} & & \tilde{S}' \ar[d]^{\pi'} \\ S \ar@{-->}[rr]^f & & S'}\] 
We still denote by $E_0$ the strict transform of $B_S$ in $X$ or in  $\tilde{S}$.  Remember (see \cite[th. 5.2 p.410]{Har}) that in general if $h: M \to M'$ is a  birational map between normal surfaces, and $p \in M$ is a base point of $h$, then there exists a curve $C \subset M'$ such that $h^{-1}(C) = p$. This implies that at every step of the resolution $\sigma$ of $f$ there is only one base point, which is the preimage of $B_{S'}$. Thus the last blow-up of the sequence $\sigma$ produces a  divisor which is the strict transform of $B_{S'}$, and  the last blow-up of the sequence $\sigma'$ produces $E_0 = B_S$.  That is to say the curve $E_0$ in  $X$ can be contracted to a smooth point. Therefore $E_0$ is  rational (as are all the other components of $B_X$, by construction), and  in $\tilde{S}$ we have $E_0^2 \ge 0$ because the self-intersection of $E_0$ should become $-1$ in $X$ after a (non-empty) sequence of blow-ups whose at least the first is located on $E_0$. Furthermore,  after the contraction of $E_0$ from $X$ the boundary is still a simple normal crossing divisor. In consequence,  $E_0$ admits at most two neighboring components in  $B_X$. This implies that $B_S$ supports at most two singularities, and if $B_S$ contains exactly two singularities, then the base  point of $f$ must coincide with one of these singularities.\\ 

\subsubsection*{Admissible class.}
We suppose given a smooth quasi-projective surface  $V$. In view of the observations above, it is natural to define the class of admissible surfaces as the set of pairs $(S,B_S)$ with the following properties:
\begin{itemize}
  \item $S$ is a projective compactification of $V$, that is we have a fixed isomorphism $S \setminus B_S \stackrel{\sim}{\rightarrow} V$;
  \item $B_S = \sum E_i$ is a reduced divisor with each $E_i$ isomorphic to $\cpo$;
  \item $(S,B_S)$ admits only dlt singularities;	
  \item If $B_S$ is irreducible then $S$ admits at most two singularities. 
\end{itemize}
We allow the possibility of a reducible boundary mainly in order to include the  surfaces $Z_i$ with two boundary components that appear in the theorem. In this case, we will observe in the course of the demonstration that each boundary component supports at most one singularity.

\begin{rem} \label{Gchains} The class of admissible surfaces we just defined  contains in  particular the class of affine surfaces that admit a compactification by a  chain of smooth rational curves, which has been studied by  Danilov and Gizatullin \cite{GD1, GD2}. Our theorem thus applies to these surfaces. Indeed, each surface of this kind admits at least one  compactification by a chain of rational curves $C_0,C_1,\ldots, C_r$, $r\geq 1$,  whose self-intersections are respectively $0,a_1,\ldots, a_r$, where $a_1 \le -1$ and $a_i \le -2$ for all $i = 2,\cdots,r$. After contracting  the curves $C_1,\ldots, C_r$, we obtain an admissible compactification $S$ with an irreducible boundary $B_S=C_0$.  These surfaces always admit a very rich automorphism group. In particular,  it acts on the surface with an open orbit of finite complement (see \cite{G}). \\
\end{rem}

\subsection{Proof of the theorem}
\label{section:preuve}
As above, let $\pi:\tilde{S}\rightarrow S$ and $\pi':\tilde{S}' \rightarrow S'$ denote the minimal resolutions of the singularities and let $\tilde{S}\stackrel{\sigma}{\leftarrow} X \stackrel{\sigma'}{\rightarrow} \tilde{S}'$ be a minimal  resolution of the  base points of $f$. The divisor  $B_X$ is then a tree of rational curves, whose irreducible components are exceptional  for at least one  of the two  morphisms  $\pi\circ\sigma$  or $\pi' \circ \sigma'$, thus they have all a strictly negative self-intersection. Since $B_X$ is a tree, there  exists a unique sub-chain $E_0,E_1,\ldots, E_n$ of $B_X$ such that $E_0$ and $E_n$ are the strict transforms of  $B_S$ and $B_{S'}$ respectively. The minimality hypotheses imply that $E_0$ and $E_n$ are the only irreducible components with self-intersection $-1$ in $B_X$. The demonstration proceeds by induction on the number  $n$ of components in the chain joining the strict  transforms of $B_S$ and $B_{S'}$, which will also be the number of  links necessary to factorize $f$. 

We use the same notation for the curves $E_i$, $i=0,\ldots,n$ and their images or strict transforms in the different surfaces that will come into play. The self-intersection of $E_0$ is positive in $\tilde{S}$ by  hypothesis.  By definition of the resolution $X$, the  divisor $E_1$ is produced by blowing-up successively the base points of $f$ as long as they lie on  $E_0$, $E_1$ being the last divisor produced by this process. Let $Y\rightarrow \tilde {S}$ be the intermediate surface thus obtained. By construction, the image of the curves contracted by the induced birational morphism  $X \to Y$  are all located outside  $E_0$. The self-intersections of $E_0$ in $X$ or in $Y$ must in particular be equal (to $-1$). The divisor $B_Y$ is then a chain that looks as in figure \ref{fig:proof}.  
The wavy curves labelled ``Sing'' correspond to the (possible) chains of rational curves obtained by desingularization of $S$, and the wavy curve labelled  ``Aux'' corresponds to the (possible)  chain of auxiliary rational curves, each one with self-intersection $-2$, obtained by resolving the base  points of $f$ before getting  $E_1$. 
\begin{figure}[ht]
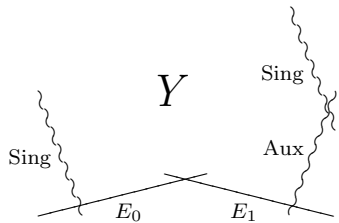

 $$\dessinpreuve$$
 \caption{The boundary divisor of $Y$.} \label{fig:proof}
\end{figure}

The lemma \ref{lem:KBneg} ensures that by running the $K+B$ MMP we can successively contract all the components of the boundary  $B_Y$ with the exceptions of $E_0$ and $E_1$. Indeed at each step, the extremities of the boundary chain support at most one singularity and thus are $K+B$ extremal, with negative self-intersection. This implies that they give rise to divisorial extremal contractions. We note $(Z,E_0+E_1)$  the dlt  pair obtained from the pair $(Y,B_Y)$ by this process.  
 $$\mygraph{
!{<0cm,0cm>;<1cm,0cm>:<0cm,1cm>::}
!{(2,3.3)}*+{X} ="X" !{(2,2.1)}*+{Y} ="Y" !{(2,0.6)}*+{Z} ="Z"
!{(0,0)}*+{S} ="S" !{(4,0)}*+{S'} ="S'" 
!{(0,1.5)}*+{\tilde{S}} ="Stilda" !{(4,1.5)}*+{\tilde{S'}} ="S'tilda"
"X"-@{->}_{\sigma}"Stilda" "X"-@{->}^{\sigma'}"S'tilda" "X"-@{->}"Y" 
"Y"-@{->}"Stilda" "Y"-@{->}"Z" "Z"-@{->}"S"
"Stilda"-@{->}_{\pi}"S" "S'tilda"-@{->}^{\pi'}"S'"
"S"-@{-->}_f"S'"
}$$
 
By construction, $Z$ dominates $S$ via the  divisorial contraction of the $K+B$ extremal curve $E_1$. Again by the lemma \ref{lem:KBneg}, $E_0$ is $K+B$ extremal, with self-intersection strictly negative in $Z$. So there exists a $K+B$ divisorial extremal contraction  $Z\rightarrow S_1$ contracting exactly $E_0$.  We obtain the first expected link and the map $f:S\dashrightarrow S'$ factorizes via a birational map $f_1:S_1\dashrightarrow S'$ for which it is straightforward to check that the length of the chain defined at the beginning of the proof is equal to $n-1$.  

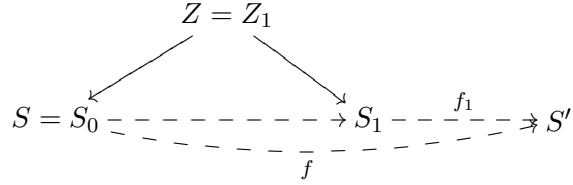
\[\xymatrix{ & Z = Z_1 \ar[dl] \ar[dr] \\ S = S_0 \ar@/{}_{1pc}/@{-->}[rrrr]_{f} \ar@{-->}[rr] & &  S_1 \ar@{-->}[rr]^{f_1} & & S'}\]
We conclude by induction that we can factorize $f$ by exactly $n$ links.\findem\\

\begin{lem}
\label{lem:KBneg}
Let $(S,B_S)$ be an admissible surface. 
   \begin{enumerate}
        \item If $C \subset B_S$ is an irreducible curve with only one neighboring component in $B_S$ and  supporting at most one singularity of $S$, then $(K_S+B_S).C <0$. 
        \item Let $C \subset B_S$ be a curve supporting exactly one singularity $p$ of $S$, and note $\overline{C}$ the strict transform of $C$ in the minimal resolution of $p$. If  ${\overline{C}}^2<0$ then $C^2<0$.  
    \end{enumerate}
\end{lem}
\begin{proof}

  \indent (1) If  $C$ does not support any singularity of $S$, we have $$(K_S+B_S)\cdot C=(K_S+C)\cdot C +1=-2+1=-1$$ by  adjunction. Otherwise, let $\pi:\overline{S}\rightarrow S$ be a minimal resolution of the unique singularity supported by $C$. Let $E$ be the unique exceptional curve of $\pi$ that meets the strict transform $\overline{C}$ of $C$.  We write $K_{\overline{S}} + B_{\overline{S}} = \pi^* (K_S+B_S) + aE + \reste$, where  $\reste$  (here and further in the proof) denotes an exceptional divisor for $\pi$, whose support does not meet $\overline{C}$, and all of whose coefficients may vary. We have then 
\begin{eqnarray*}
 (K_S+B_S)\cdot C & = & (K_S+C) \cdot C + 1  =  \pi^*(K_S+C)\cdot \pi^*C +1 \\
         			& = & (K_{\overline{S}} + \overline{C} +(1-a)E + \reste) \cdot \overline{C} +1  =  \pi^*(K_S+C)\cdot \overline{C} + 1  \\
           			& = & (K_{\overline{S}}+\overline{C}).\overline{C}  + 1 - a + 1   =  -2 + 2 - a \\
             			& = & -a <0  
\end{eqnarray*}
because $(S,B_S)$  dlt pair implies  $a>0$.\\
 \indent (2)  Let  $\pi:\overline{S}\rightarrow S$ be as above. We write  $\overline{C}=\pi^*C-bE - \reste$ and  $K_{\overline{S}}=\pi^*K_S-cE -\reste$ where $c\geq0$ (otherwise $p$ would be a smooth point) and $b>0$. We have 
\begin{eqnarray*}
   C^2=(\pi^*C)^2            & =&  (\overline{C} + bE + \reste)^2 ={\overline{C}}^2+2b+(bE+\reste)^2\\
   			                & = & {\overline{C}}^2+2b+(bE+\reste)\cdot(\pi^*C-\overline{C}) ={\overline{C}}^2+2b-bE\cdot\overline{C}\\
			       	        & = &  {\overline{C}}^2+b 
\end{eqnarray*}             
On the other hand, in the logarithmic ramification formula above, we have  $a=1-b-c>0$ because $(S,B_S)$ is a dlt pair. So $1>b$, and therefore  $C^2<{\overline{C}}^2+1$: this  gives the assertion of the lemma. 
\end{proof}
\mbox{ }

\begin{rem} \label{rem:2voisins}
The lemma contains what is strictly necessary for the demonstration of the theorem. Nevertheless an easy refinement of the first assertion of the lemma leads to a more precise  caracterisation of a $K_S+B_S$ extremal boundary component $C$ of a dlt pair $(S,B_S)$. In particular: \\

\textit{A curve $C$ in $B_S$ with at least two neighboring components in $B_S$ can never be  $K_S+B_S$ extremal.}\\

\begin{proof}
If we note $n$ the number of neighbors of $C$ in  $B_S$ and $p_1,\ldots , p_r$ the singular points of $S$ supported along $C$, the same argument as in the proof of the lemma shows that $$ (K_S+B_S)\cdot C  =  -2 + \sum_{i=1}^r (1-a_{i,1}) +n,$$ where for all $i=1,\ldots, r$,   $a_{i,1}>0$ is the log discrepancy of the unique  exceptional divisor $E_{i,1}$ in the minimal resolution of the singular point  $p_i$ that meets the strict transform of $C$. 
As above we can show that the log discrepancies $a_{i,j}>0$ of the irreducible components  $E_{i,1},\ldots , E_{i,r_i}$ of the chain of exceptional divisors in the minimal resolution of $p_i$ are all strictly less than $1$ (write  $a_{i,j}$ as  $a_{i,j} = 1 - b_{i,j} -c_{i,j}$ as in the end of the proof of the lemma). If we note  $\pi:\overline{S}\rightarrow S$ the surface obtained by taking the minimal resolution of all the singular points $p_i$, we have 
\begin{eqnarray*}
  a_{i,1}E^2_{i,1}+a_{i,2} & = & (\pi^*(K_S+B_S)+\sum_{i,j} a_{i,j}E_{i,j})\cdot E_{i,1} \\
   					&= & (K_{\overline{S}}+B_{\overline{S}})\cdot E_{i,1} =(K_{\overline{S}}+\overline{C}+E_{i,1}+E_{i,2})\cdot E_{i,1} \\
					& = & -1+ E_{i,2}\cdot E_{i,1} 
\end{eqnarray*}   
where, by convention, $a_{i,2}=E_{i,2}\cdot E_{i,1}=0$ if $p_i$ admits a resolution by a blow-up with a unique exceptional divisor $E_{i,1}$.  In all cases we have  $ a_{i,1}\leq1/2$  because  $E^2_{i,1}\leq -2$ and $a_{i,2}<1$ . Finally we get   
$$  n-2+r>(K_S+B_S)\cdot C  \geq n -2 + \frac{r}{2}.$$ 
For $n=2$, this gives the assertion above. For $n=1$ and $r = 1$ we get again the first assertion of the lemma. For $n= 0 $ and $r = 2$ we obtain: if $B_S = C$ is irreducible and  supports two  singularities then $(K+B).C <0$  (but $C$ may not be $K+B$ extremal).  
\end{proof}

The inequalities above do not  \textit{a priori} exclude the possibility for a curve $C$ without a neighbor and supporting three  singularities to be  $K_S+B_S$ extremal. But in this case the singularities could no longer be of an arbitrary type. For instance, the computations above  shows that an isolated component $C$ supporting exactly three $A_{3,1}$ singularities  (i.e., each one admits a resolution by a unique exceptional rational curve with self-intersection $-3$) satisfies $(K_S+B_S)\cdot C=0$. On the other hand, let us remember from the paragraph \ref{section:cadre} that if $S$ admits more than two singularities then any birational map $S \dashrightarrow S'$ induced by an automorphism of $V$ is in fact a morphism.  

We should finally remark that neither the lemma nor the above argument tell something about the possible $K+B$ extremal curves that do not belong to the boundary: in the example given in the introduction, we had four $K+B$ extremal rays, only two of which were within the boundary.\\
\end{rem}

\section{Comments and complements} \label{par:comments}

First we discuss additional properties of our algorithm with respect to the work by Bruno and Matsuki \cite{BM}. We also obtain as a corollary of our main theorem a description of the automorphism group of $V$ when $V$ admits a smooth admissible compactification. When $V = \C^2$ this is essentially the classical theorem of Jung and Van der Kulk (see example \ref{expl:c2}), and when $V$ is the complement of a section in a surface of Hirzebruch we obtain a new way of expressing some results of Danilov and Gizatullin (see example \ref{expl:+4}).

\subsection{Relations with the log Sarkisov program}

  Let us first consider again the subchain $E_0, \cdots, E_n$ of rational curves in the boundary $B_X$ of $X$ defined in the proof. Lemma  \ref{lem:KBneg} guarantees that all the irreducible components of $B_X$ except the ones contained in that chain can be successively contracted by a process of the  $K+B$ MMP. The surface $W$  obtained by this procedure has boundary $B_W = \sum_{i=0}^nE_i$ and dominates both $S$ and $S'$ by a sequence of $K+B$  divisorial contractions: see figure \ref{fig:KBMMP} (note that $W$  is in general singular). 
\begin{figure}[ht]
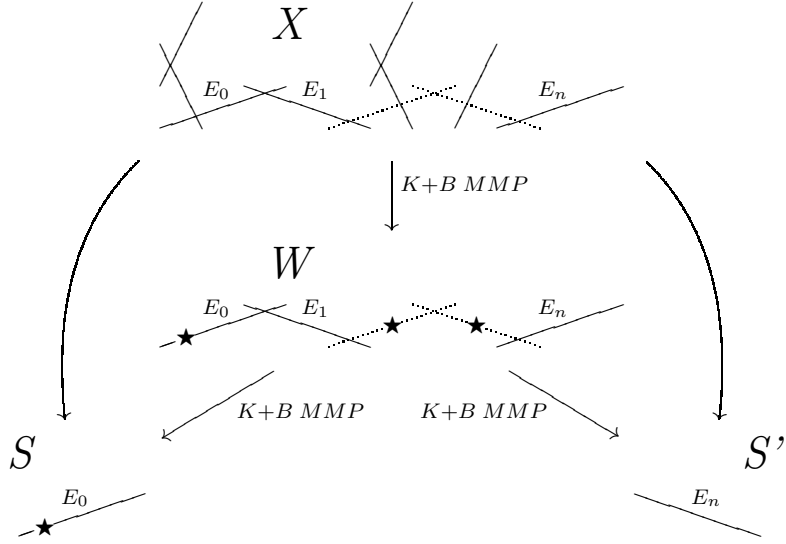

$$\mygraph{
!{<0cm,0cm>;<\size,0cm>:<0cm,\size>::}
!{(7.5,10)}*+++{\dessinresolutioncomplete} ="Comp"
!{(7.5,4.5)}*++{\dessinresolutionKplusB } ="K+B"
!{(0,0)}*++{\dessinEzerosing} ="E0"
!{(15,0)}*++{\dessinEn} ="En"
"Comp"-@{->}^{K+B\; MMP}"K+B" "Comp"-@/_2cm/@{->}"E0" "Comp"-@/^2cm/@{->}"En" 
"K+B"-@{->}^{\quad K+B\;MMP}"E0" "K+B"-@{->}_{K+B \;MMP\quad}"En" 
}$$
\caption{Log-MMP relation between $S$ and $S'$.}\label{fig:KBMMP}
\end{figure}
It follows that the pairs $(S,B_S)$ and $(S',B_{S'})$  considered in the theorem are always log MMP related in the sense of Matsuki \cite[p.128]{Mat}.  This leads to the following  result. \\

\begin{prop}
\label{prop:max}
 The birational morphism  $Z \to S$ with exceptional divisor $E_1$ constructed in the proof of the theorem is a maximal extraction . 
\end{prop}

\begin{proof}
 A maximal extraction  (see \cite[prop. 13-1-8]{Mat} and \cite[p.485]{BM} for the logarithmic case) is obtained from a surface which dominates $S$ and $S'$ by a process of the $K+B$-MMP. So we may start with the surface $W$ just constructed. The precise procedure consists in running a $K+B+\frac{1}{\lambda}H$-MMP over $S$, where $\lambda$ and $H$ have been defined in the introduction. The crucial observation is that each extremal divisorial contraction of this log MMP is also a one of the genuine  $K+B$-MMP (this is obvious for surfaces as $H$ is nef, but this property actually holds in any dimension).  The fact that we are running a MMP over  $S$ guarantees that the only curves affected by the procedure are contained in the boundary.  By virtue of lemma  \ref{lem:KBneg} and remark \ref{rem:2voisins}, the only $K+B$ extremal rays contained in a chain are its terminal components. So there exists a unique sequence of  $K+B$ divisorial contractions from $W$ to $S$. It follows that the last one $Z\rightarrow S$, which has for exceptional divisor $E_1$,  is a maximal extraction.
 \end{proof}

The Sarkisov program has been initially designed as an algorithm to factorize birational maps between a class of varieties as simple as possible in the context of the  Minimal Model Program, namely,  Mori fiber spaces. Here we replaced Mori fiber spaces by another class of very simple objects : dlt pairs $(S,B_S)$ with an irreducible boundary $B_S$.  It may happen that certain pairs $(S,B_S)$ also admits a structure of a log Mori fiber space. This holds for instance for admissible compactifications of the affine plane  $\C^2$ by a smooth rational curve. Indeed, the latter admits a trivial structure of Mori fiber space $S \to pt$,  due to the fact that their Picard group is of rank one. Using Proposition \ref{prop:max} above, it is not difficult to check that for such surfaces our algorithm coincides with the log Sarkisov program of Bruno-Matsuki.  Furthermore, the factorization enjoys the following property. \\

\begin{prop}
Let $S$ and $S'$ be admissible surfaces equipped with a structure of trivial Mori fiber space and let  $f:S \dashrightarrow S'$ be a birational map extending  an automorphism of $V$.  Then each link of our algorithm strictly decreases the log Sarkisov degree $\mu$.
\end{prop}

\begin{proof}
Letting  $\mu_S$ and $\mu_{S_1}$ be the degrees of  $f: S \dashrightarrow S'$ and  $f_1:S_1 \dashrightarrow S'$ respectively (see the proof of the theorem for the notation), the ramification formulas read 
\begin{eqnarray*}
K_Z+B_Z + \frac{1}{\mu_S} H_Z  &=& g^*(K_S+B_S+\frac{1}{\mu_S} H_S) + (c-b-\frac{m}{\mu_S})E_1 \\
&=& {g'}^*(K_{S_1}+B_{S_1}+\frac{1}{\mu_S} H_{S_1}) + \star E_0,
\end{eqnarray*}
where  $\lambda = m/(c-b)$ is the maximal multiplicity and where  $\star$ is a coefficient which plays no role in the sequel. By definition, $ K_S+B_S+\frac{1}{\mu_S} H_S \equiv 0$ and due to the fact that all the curves in $S_1$ are numerically  proportional,  we have $\mu_{S_1} < \mu_S$ provided that  $(K_{S_1}+B_{S_1}+\frac{1}{\mu_S} H_{S_1}).C<0$ for an arbitrary curve in $S_1$.  Since $f$ is not an isomorphism the logarithmic version of the N\oe ther-Fano criterion \cite[prop. 13-1-3]{Mat} guarantees that  $\lambda > \mu_S$.  Thus, given a  curve  $C \subset Z$ intersecting $E_1$ but not  $E_0$, one checks that its image in $S_1$ satisfies
\begin{eqnarray*}
(K_{S_1}+B_{S_1}+\frac{1}{\mu_S} H_{S_1}).C  &= & ({g'}^*(K_{S_1}+B_{S_1}+\frac{1}{\mu_S} H_{S_1}) + \star E_0).C \\
&=& (c-b) \underbrace{( 1-\lambda/\mu_S)}_{<0}
 \underbrace{E_1.C}_{>0} < 0
\end{eqnarray*}
as desired. 
\end{proof}

\subsection{Isomorphisms between surfaces with a smooth compactification}
\label{sec:smooth}

In this paragraph we suppose that we have an isomorphism $f:V \to V'$ between two surfaces that admit compactifications by \textit{smooth} admissible surfaces $S$ and $S'$. It is known that if $V$ is affine and admits a smooth admissible compactification, then the pair $(S, B_S)$ is one of the following: $(\cpd, \mbox{a line})$, $(\cpd, \mbox{a conic})$, or $(\F_n, \mbox{an ample section})$. Furthermore, in the latter case, the isomorphism class of $V$ only depends on the self-intersection of $B_S$, and not on the ambiant $\F_n$ nor on the particular section chosen (see \cite{GD1, GD2} or \cite{FKZ}). 

Now let $S_i$ be one of the singular admissible compactification of $V$ that appear in the factorization of $f$. We prove in the next lemma that $S_i$ can only have one singularity, so the following definition makes sense: We say that $S_i$ has \textbf{index} $n$ if in the minimal resolution of the singularities of $S_i$ the exceptional curve which intersects the strict transform of $B_{S_i}$ has self-intersection $-n$. On the other hand if $S_i$ is any smooth compactification of $V$ we say that $S_i$ has index 1. We note $\ind(S_i)$ the index of $S_i$. 

\begin{lem} \label{lem:triangle}
 Let $S_0 = S \longleftrightarrow S_1 \longleftrightarrow \cdots \longleftrightarrow S_n = S'$ be the factorization into elementary links given by applying the theorem \ref{th:main} to $f:V\to V'$, where $S, S'$ are smooth, and $B_S^2 >0$. Then:
\begin{enumerate}
 \item Each $S_i$ has at most one singularity;
 \item For all $i = 0, \cdots, n-1$ the index  $\ind(S_i)$ and $\ind(S_{i+1})$ differ exactly by 1;
 \item If $\ind(S_i) \ge 2$ and $\ind(S_{i-1}) = \ind(S_i) + 1$ then   $\ind(S_{i+1}) = \ind(S_i) - 1$.
\end{enumerate}
\end{lem}

\begin{proof}
Suppose $S_i$ is smooth, with $B_{S_i}^2 = d >0$. Consider the surface $Y$ constructed in the proof of theorem \ref{th:main}, which contain $E_i$ and $E_{i+1}$ the strict transforms of the boundary of $S_i$ and $S_{i+1}$. Then the boundary of $Y$ contains a chain of $d$ curves with self-intersection $-2$, thus $S_{i+1}$ has only one singularity and has index 2 (see Fig. \ref{fig:triang}, (a), with $k = 2$). 

Now suppose $S_i$ has exactly one singularity and has index $k$. If the base point on $E_i = B_{S_i}$ coincide with the singularity of $S_i$, then the  boundary of $Y$ contains a chain of curves, the first one with self-intersection $-(k+1)$ and intersecting $E_{i+1}$, and all the other ones with self-intersection $-2$. Thus in this case $S_{i+1}$ has again exactly one singularity,  and has index $k+1$ (the picture is again Fig. \ref{fig:triang}, (a)).

Finally suppose $S_i$ has exactly one singularity, has index $k$ and the base point on $E_i = B_{S_i}$ does not coincide with the singularity. Then  the  boundary of $Y$ contains a chain of curves, the first one with self-intersection $-k$ and intersecting $E_i$. In this case $S_{i+1}$ has again exactly one singularity,  and has index $k-1$ (see Fig. \ref{fig:triang}, (b)).
\end{proof}

\begin{figure}[ht]
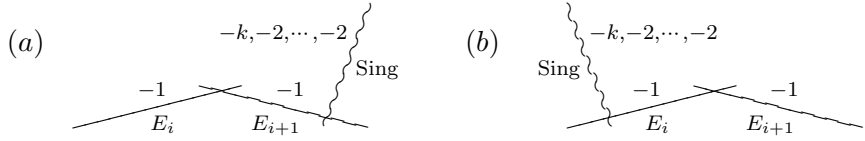
 
$$ (a) \quad \dessinTa \qquad (b) \quad \dessinTb $$
\caption{Boundary of $Y$ in the proof of lemma \ref{lem:triangle}.} 
\label{fig:triang}
\end{figure}

We say that an isomorphism $f:V \to V'$ is \textbf{triangular} \label{def:triang} if all the intermediate surfaces $S_i$ that appear in the factorization have index greater that 2, that is if all the $S_i$ except $S_0$ and $S_n$ are singular.  As an easy consequence of the lemma \ref{lem:triangle} we see that the index of the $S_i$ grows until it reaches its maximum $\ind(S_k) = k-1$, and then the index goes down until we reach the smooth surface $S_{n}$ (so we have $n = 2k -2$). We say that $k$ is the degree of the triangular isomorphism $f$.  Note that we had to suppose  $B_S^2 >0$  in the lemma \ref{lem:triangle} otherwise all the $S_i$ are  smooth; in this case (for instance for $\cpo \times \cpo$ minus a fiber)  the notion of a triangular automorphism  coincide with the notion of a link.

If $S$ is smooth, with $B_S^2 = d \ge 0$, to each  $p \in B_S$ we can associate a rational pencil $\mathcal{P}$ whose $B_S$ is a member: take the curves in the linear system $|B_S|$ with a local intersection number  with $B_S$ at $p$ equal to $d$. The resolution of the base locus of $\mathcal{P}$ gives a surface $\hat{S}$ and a fibration $ \hat{S} \to \cpo$ with a unique singular fiber, and the last exceptionnal divisor produced is a section $C$ for this fibration with self-intersection $-1$. One can prove (see \cite{BD}, section 2) that by a sequence of contractions one can construct a morphism from $\hat{S}$ to $\F_1$, where the strict transform of $C$ is the exceptional section. We will use this point of view in the example \ref{expl:+4}, see Fig. \ref{fig:models}.  
 
One can perform a sequence of elementary transformations starting from $\hat{S}$: choose a point on the fiber $B$ which is the strict transform of $B_S$, blow-up this point to produce $B'$, then blow-down $B$. The intermediate variety corresponds to the surface $Y$ in the proof of lemma \ref{lem:triangle}. A sequence of such elementary transformations clearly preserves the fibration induced by $\mathcal{P}$. As in lemma \ref{lem:triangle} we see that the self-intersection of the section $C$ goes down until reaching its minimum $-k$, then it grows until reaching $-1$ again. We call such a sequence a \textbf{fibered modification} of degree $k$.
 
\begin{lem}\label{lem:relation}
\begin{enumerate}
 \item If $f :V \to V'$ is a triangular isomorphism, and $\mathcal{P}$, $\mathcal{P}'$ are the pencil on $S, S'$ associated with the base point of $f$, $f^{-1}$, then $f$ sends the pencil $\mathcal{P}$ on $\mathcal{P}'$;
 \item Conversely if $f : V \to V'$ sends  $\mathcal{P}$ on $\mathcal{P}'$, then $f$ is triangular;
 \item If $f :V \to V'$  and $g :V' \to V''$ are two triangular isomorphisms such that the base point of $f^{-1}$ and $g$ coincide, then  the composition $g \circ f : V \to V''$ is still a triangular isomorphism.  
\end{enumerate}
\end{lem}

\begin{proof}
 \begin{enumerate}
\item This follows from the observation that if $f$ is a triangular isomorphism of degree $k$, then $f$ can be factorized as the $d$ blow-ups giving the resolution of the base locus of $\mathcal{P}$, followed by a fibered modification of degree $k$, and followed finally by $d$ blow-downs.
\item First the base locus of $f$ has to be contained in the base locus of $\mathcal{P}$, thus the resolution of $f$ began by the $d$ blow-ups corresponding to $\mathcal{P}$. Thus the factorization of $f$ began with a triangular isomorphism $f_1$. But then $f$ has to be equal to $f_1$, otherwise the base point of $f \circ f_1^{-1}$ would have to be different from the base point of $f(\mathcal{P})$: a contradiction.
\item By the assertion (1) $f(\mathcal{P}) = \mathcal{P}'$ and $g(\mathcal{P}') = \mathcal{P}''$, thus $g\circ f (\mathcal{P}) = \mathcal{P}''$ and by (2) we obtain that $g\circ f$ is triangular.
\end{enumerate}
\end{proof}

If $f:V \to V'$ and $g:V' \to V''$ are isomorphisms and $S,S',S''$ are smooth compactifications of $V,V'$ and $V''$ such that the base point of $f^{-1}$ and $g$ coincide,  we say that $f$ and $g$ are in \textbf{special position}, otherwise we say that they are in \textbf{general position}. 
Then the following is a simple consequence of our main theorem.

\begin{prop}\label{prop:main}
Suppose $V, V'$ are two quasi-projective surfaces which admit smooth admissible compactifications $S, S'$.
 \begin{enumerate}
  \item Any isomorphism $f:V\to V'$ is a composition of triangular isomorphisms $f = f_r \circ \cdots \circ f_1$, with each pair $f_{i}, f_{i+1}$ in general position.
 \item \label{point:rel} If we consider a composition $g = g_r \circ \cdots \circ g_1$ of triangular isomorphisms with any two successive $g_i$ in general position, then the factorization of $g$ given by the theorem \ref{th:main} is equal to the concatenation of the factorizations of each $g_i$. 
 \end{enumerate}
\end{prop}

\begin{proof}
\begin{enumerate}
\item Let  $S_0 = S \longleftrightarrow S_1 \longleftrightarrow \cdots \longleftrightarrow S_n = S'$ be the factorization into elementary links given by the theorem \ref{th:main}. Consider the minimal index $i \ge 1$ such that $S_i$ is smooth. Then by the lemma \ref{lem:triangle} the sequence of links  $S \longleftrightarrow S_1 \longleftrightarrow \cdots \longleftrightarrow S_i$ gives a triangular isomorphism. We conclude by induction on the number of links in the factorization of $f$.
\item This assertion follows from the following claim: if $f= h \circ g$ where $g :V \to V''$, $h:V'' \to V'$ are in general position, then the factorization of $h \circ g$ is the concatenation of the factorization of $g$ and $h$. This is an easy observation once we remark that if we  start from a smooth compactification $S''$ of $V''$ and we simultaneously resolve the base points of $g^{-1}$ and $h$ (this is possible because we suppose they are distinct), we obtain a surface that is the minimal resolution of $f$.
\end{enumerate}
\end{proof}

One can think of the proposition \ref{prop:main} as a kind of presentation by generators and relations. In particular the assertion (2) says that there is no relation except the trivial ones given by lemma \ref{lem:relation}, (3). However, even if $S = S'$ and $f$ is an automorphism of $V$, in general the triangular isomorphims $f_i$ are not birational transformations between two isomorphic compactifications of $V$. We give an example in paragraph \ref{expl:+4} of a situation where two kinds of compact models ($\F_0$ and $\F_2$) come into play. If one insists in having generators that live on a chosen model, one has to chose some way to pass from each possible model to the chosen model. This is what is done in \cite{GD2}, where the relations are expressed in terms of almagamated products. However, one loses the simplicity of the statement (\ref{point:rel}) in the proposition \ref{prop:main}. 

In their pioneering work \cite{FM}, Friedland and Milnor noticed that the theorem of Jung allows to obtain some normal forms up to conjugacy for automorphisms of $\C^2$. This was the starting point for an exhaustive study of the possible dynamical behaviour of an automorphims of $\C^2$ (see \cite{BS8} and references therein). As a consequence of proposition \ref{prop:main}, we obtain normal forms for the automorphisms of a surface $V$ which admits a smooth admissible compactification.

\begin{cor}\label{cor:FMlike}
Let $V$ be a quasi-projective surface which admits a smooth admissible compactification $S$, and let $f : S \dashrightarrow S$ be a birational map which induces an automorphism on $V$. Then there exists another smooth admissible compactification $S'$ of $V$ and a map $\varphi:S \dashrightarrow S'$ such that $g = \varphi f \varphi^{-1} :S' \dashrightarrow S'$ has one of the following two properties:
\begin{enumerate}
 \item the pair $g,g$ is in general position (that is, $g$ and $g^{-1}$ have distinct proper base point);
 \item either $g$ is biregular or $g$ is a triangular automorphism with the pair $g,g$ in special position.
\end{enumerate}
\end{cor}

\begin{proof}
 Suppose $f$ is not biregular, and consider the factorisation $f = f_r \cdots f_1$ into triangular isomorphisms given by the proposition \ref{prop:main}, (1). If the pair $f,f$ is in special position and $f$ is not triangular (that is, $r \ge 2$), then we consider the conjugate $f' = f_r^{-1} f f_r$. By hypothesis $f_1$ and $f_r$ are in special position so $f_1 f_r$ is a triangular isomorphism, in consequence  $f' = f^{r-1} \cdots f_2 (f_1 f_r)$ admits a factorization into $r-1$ triangular isomorphisms, and we are done by induction on the number of triangular isomorphisms necessary to factorize $f$. 
\end{proof}

In particular if $g$ satisfies property $(1)$ in the conclusion of the corollary \ref{cor:FMlike}, and if $n$ links are necessary to factorize $g$, then exactly $n|k|$ links are necessary to factorize the iterate $g^k$ (where $k \in \Z$). Such a map is similar to a composition of generalized H\'enon maps, so one can expect that these maps will always present a chaotic dynamical behaviour. On the other hand, any finite automorphism of $V$, or any one-parameter flow of automorphisms of $V$, will always correspond to the case (2) in the corollary \ref{cor:FMlike}. 

\section{Examples}

In the first three paragraphes of this section, we apply our algorithm to describe the automorphism groups of certain affine varieties admitting a smooth admissible compactification.  Since we hope  that a possible generalization of our results in higher dimension could help to understand the structure of the automorphism group of the affine space $\C^3$, we first check that we can recover the structure of the automorphism group of $\C^2$  using our algorithm. Then we complete the example of an affine quadric surface $\CP^1\times \CP^1\setminus D$ which served as a motivation in the introduction, recovering the description given in \cite{Lamyquad}. Finally, we consider the more subtle situation of the complement of a smooth rational curve in $\CP^1\times \CP^1$  with self-intersection $4$. We give a complete description in terms of isomorphisms preserving certains pencils of rational curves which is simpler than the one obtained by Danilov and Gizatullin (compare with \cite[\S 7]{GD2}). 

In the last paragraph we illustrate the notion of chain reversion. We refer to \cite{BD} for a detailed study of this situation.

\subsection{Automorphisms of $\C^2$}
\label{expl:c2}
Here we derive Jung's Theorem from the description of the triangular maps   which appear in the decomposition of an automorphism of $\C^2$. We also refer the reader to \cite{LamyJung} and \cite{Mat}, which contain proofs of Jung's Theorem derived from the philosophy of (log) Sarkisov program (but not formulated in the language of Mori Theory in the former).  
Since $(\cpd,L)$,  where $L$ is a line, is the only smooth admissible compactification of $\C^2$, our algorithm leads to a decomposition of an arbitrary
polynomial automorphism $f$ of $\C^2$ into a sequence of triangular birational maps $j_i:\cpd\dashrightarrow\cpd$ (see section \ref{sec:smooth} for the definition of a triangular map). Let $d_i\geq 2$ be the degree of $j_i$. Each of these maps are obtained as a sequence of $2d_i-2$ links of the form 
 $$ \cpd \longleftrightarrow \cpd(2) \longleftrightarrow \cpd(3) \longleftrightarrow \cdots \longleftrightarrow \cpd(d_i) \longleftrightarrow \cdots \longleftrightarrow \cpd(2) \longleftrightarrow \cpd$$
where each  $\longleftrightarrow$ denotes an elementary link,  and where  $\cpd(d)$ denotes the weighted projective plane $\cpd(d,1,1)$, obtained from the Hirzebruch surface $\mathbb{F}_d\rightarrow \mathbb{P}^1$ by contracting the section with self-intersection $-d$. The automorphism of $\C^2\subset\cpd$ induced by $j_i$ extends to an automorphism  $\sigma_i$ of $\cpd(d_i)$, and  the above decomposition can be thought as a conjugation $j_i = \varphi^{-1} \sigma_i \varphi$, where $\varphi: \cpd \dashrightarrow \cpd(d_i)$ denotes the birational map induced by the identity on $\C^2$.  This implies in particular that each $j_i$ considered as an automorphism of $\C^2$  maps a pencil $\mathcal{P}$ of parallel lines in $\C^2$ isomorphically onto a second pencil of this type. So there exists affine automorphisms $a_{i}, a_{i+1}$ of $\C^2$ such that $a_i j_i a_{i+1}$ preserves a fixed pencil, say the one given by the second projection $\C^2\rightarrow \C$. Thus each $a_i j_i a_{i+1}$ is a triangular automorphism of $\C^2$ in the usual sense, and via the Prop. \ref{prop:main} we recover the description of the automorphism group of $\C^2$ as the amalgamated product of its subgroups of affine and triangular automorphisms over their intersection. 

\subsection{Automorphism of the affine quadric surface $\CP^1\times \CP^1\setminus D$ }

 We consider again the birational map  $f:\CP^{1}\times\CP^{1}\dashrightarrow\CP^{1}\times\CP^{1}$ given in the introduction. With the notation of the theorem's proof, the union of the boundary of the resolution  $X$ constructed in the introduction and of the strict transforms $D_{+}$ and $D_{-}$  of the members of the standard rulings on $\mathbb{P}^1\times \mathbb{P}^1$ passing through the base point $p=\left(\left[1:0\right],\left[1:0\right]\right)$ of $f$  is described by the figure \ref{fig:quad1}.
\begin{figure}[ht]
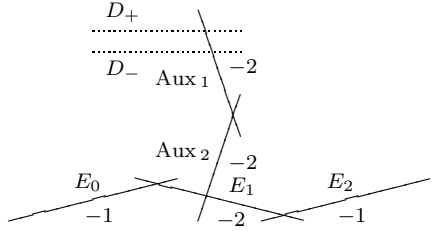

 $$\dessinresolutionbasicbis$$
\caption{Resolution of $f$.} \label{fig:quad1}
\end{figure}

Our algorithm gives a factorization  $f:\CP^{1}\times\CP^{1}\leftrightarrow S_{1}\leftrightarrow \CP^{1}\times\CP^{1}$, where the surface $S_1$  is obtained from $X$ by contracting the curves  $E_{0},E_{2}$ onto smooth points and the two auxiliary curves onto a singularity of type $A_{3,2}$ supported on $E_1$.  The Picard group of $S_1$ is isomorphic to $\mathbb{Z}^2$,  generated by the strict transforms of  $D_{+}$ and $D_{-}$, and the latter also generate the only $K+B$ extremal rays on $S_1$.  One checks further that $S_1$ dominates  $\CP^{2}$ via the divisorial contraction of any of these two curves.  So, in contrast with the situation in the log-Sarkisov program of Bruno-Matsuki,   $S_1$ does not admit a Mori fiber space structure. 

     We may identify the affine quadric  $V=\left\{ w^{2}+uv=1\right\} \subset\C^{3}$ with  $\CP^{1}\times\CP^{1}\setminus D$ via the open
immersion  $\left(u,v,w\right)\mapsto\left(\left[u:w+1\right],\left[u:1-w\right]\right)$. Then, the automorphism of  $\CP^{1}\times\CP^{1}\setminus D$
induced by  $f$ coincides with the unique automorphism of $V$ lifting the triangular  automorphism
$\left(u,w\right)\mapsto\left(u,w+u^{2}/2\right)$ of $\mathbb{C}^2$ via the birational morphism
 $V\rightarrow\C^{2}$, $\left(u,v,w,\right)\mapsto\left(u,w\right)$.  The latter triangular automorphism uniquely extends to a biregular automorphism  $\phi$ of the Hirzebruch surface $\F_{2}\rightarrow\CP^{1}$ via the open immersion of $\C^2$ in $\mathbb{F}_2$ as the complement of the union of a fiber $E_1$ and of the section $\aux_2$ with self-intersection $-2$.  In turn, the birational morphism  $V\rightarrow\C^{2}$ lifts to an open immersion of $V$ into the projective surface $\bar{V}$ obtained from 
 $\F_{2}$ by blowing-up the two points  $q_{\pm}=\left(0,\pm1\right)\subset\C^{2}\subset\F_{2}$ with exceptional divisors  $D_{\pm}$  respectively. The boundary $B_{\bar{V}}$  consists in the union of the strict transforms of  $\aux_{2}$, $E_{1}$ and of the fiber  $\aux_{1}$ of $\F_{2}\rightarrow\CP^{1}$ containing the points $q_{\pm}$ (see Fig. \ref{fig:quad2}).
\begin{figure}[ht]
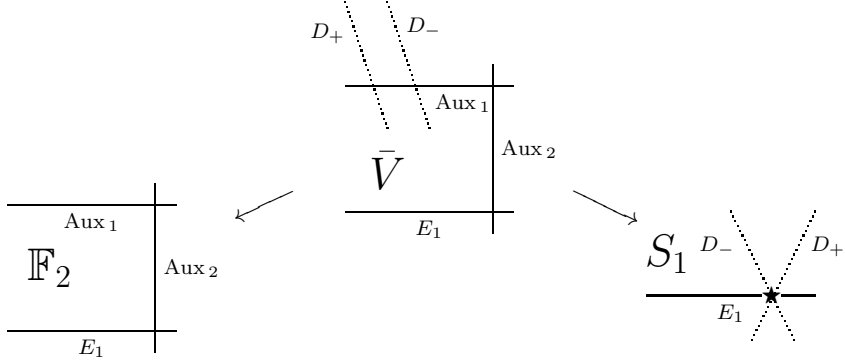

 $$\mygraph{
!{<0cm,0cm>;<1cm,0cm>:<0cm,1cm>::}
!{(0,0)  }*+{\dessinFdeux}="F2" 
!{(4.2,2)  }*+{\dessinVbar}="Vbar"
!{(8.4,0)  }*+{\dessinSun}="S1"
"Vbar"-@{->}"F2" "Vbar"-@{->}"S1" 
}$$
\caption{Sequence of blow-ups and contractions from $\F_2$ to $S_1$.} \label{fig:quad2}
\end{figure}

The automorphism  $\phi$ of $\F_{2}$ lifts to an automorphism of  $\bar{V}$ which restricts on $V$ to the automorphism induced by $f$. The latter descends to a biregular automorphism with the same property on the surface isomorphic to $S_1$ obtained from $\bar{V}$ by contracting the curves $\aux_{1}$ et $\aux_{2}$.   

By virtue of Gizatullin's result, every smooth admissible compactification $(S,B_S)$ of $V$ is isomorphic to $(\mathbb{P}^1\times \mathbb{P}^1,C)$ where $C$ is a smooth rational curve of self-intersection $2$. Applying our algorithm to an arbitrary automorphism $f$ of $V$ leads to a decomposition into triangular maps $j_i:(\mathbb{P}^1\times \mathbb{P}^1,C)\dashrightarrow (\mathbb{P}^1\times \mathbb{P}^1,C')$ with corresponding  decomposition into elementary links of the form $$ \CP^{1}\times\CP^{1}\leftrightarrow\hat{\mathbb{P}}^2(1)\leftrightarrow\hat{\mathbb{P}}^2(2)\leftrightarrow\cdots\leftrightarrow\hat{\mathbb{P}}^2(d)\leftrightarrow\cdots\leftrightarrow\hat{\mathbb{P}}^2(1)\leftrightarrow\CP^{1}\times\CP^{1}$$ where the projective surface $\hat{\mathbb{P}}^2(d)$ is  obtained  from the Hirzebruch surface  $\F_{d}$, with negative section  $C_0$,  by first blowing-up two distinct points in a fiber $F\setminus C_0$  of  $\F_{d}\rightarrow\cpo$ and then contracting  successively the strict transforms of $F$ and $C_0$. By construction,  $\hat{\mathbb{P}}^2(d)$  dominates the weighted projective plane $\cpd(d)$ via the the divisorial contraction of any of the strict transforms of the exceptional divisors of the first blow-up. 

   Since the group of biregular automorphisms of $\mathbb{P}^1\times \mathbb{P}^1$ acts transitively on the set of pairs consisting of a smooth rational curve with self-intersection $2$ and a point lying on it, we may compose $j_i$ by such an automorphism to obtain a birational transformation $a_i\circ j_i \circ a_{i-1}: \mathbb{P}^1\times \mathbb{P}^1\setminus D\dashrightarrow  \mathbb{P}^1\times \mathbb{P}^1\setminus D$ having $p=[0:1],[0:1]$ as a base point.
Its restriction to the complement of $D$ can then be interpreted as an automorphism of $V$ that preserves the fibration $\pi_u:V\rightarrow \C$ given as the restriction to $V$  of the unique rational pencil $\mathcal{P}$ on $\mathbb{P}^1\times \mathbb{P}^1$  with $p$ as a unique base point and containing  $D$ as a smooth member. One checks easily that such automorphisms come as the lifts via the morphism ${\rm p}_{w,u}:V\rightarrow \C^2$ of triangular automorphisms of $\C^2$ of the form $(w,u)\mapsto (w+uP(u),au)$, where $a\in \C^*$ and $P(u)$ is a polynomial. 
Let ${\rm Aff}$ be the subgroup of ${\rm Aut}(V)$ consisting of the restrictions of  biregular automorphisms of $\mathbb{P}^1\times \mathbb{P}^1$ preserving $D$. We conclude that ${\rm Aut}(V)$ is the amalgamated product of the group of automorphisms of the fibration  $\pi_u$ and of the group ${\rm Aff}$ over their intersection.

\subsection{ Automorphisms of the complement of a section with  self-intersection $4$ in $\F_0$ }\label{expl:+4}

Here we consider the case of an affine surface $V$ admitting
a smooth admissible compactification by a rational curve with self-intersection
$4$. According to Gizatullin's classification, the pairs corresponding
to such compactifications are either $\left(\mathbb{F}_{0},B\right)$ where
$B$ is an arbitrary smooth rational curve with self-intersection
$4$ or $\left(\mathbb{F}_{2},B\right)$ where $B$ is a section of
the $\mathbb{P}^{1}$-bundle structure $\mathbb{F}_{2}\rightarrow\mathbb{P}^{1}$
with self-intersection $4$. \\

First we review these compactifications $\left(S,B\right)$
with a particular emphasis on the existence of certain rational pencils
related with all possible triangular elementary maps that can occur
in the factorization given by proposition \ref{prop:main}. 

\begin{figure}[ht]
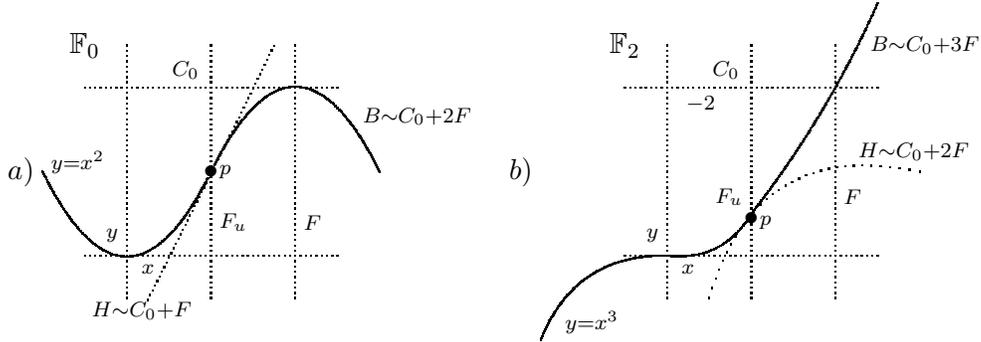

$$ a)\; \dessinDansFzero \quad b)\; \dessinDansFdeux$$
\caption{Construction of pencils in $\F_0$ and $\F_2$ (dotted curves are in $V$).}\label{fig:pencils}
\end{figure}

\subsubsection{The case $\left(\mathbb{F}_{0},B\right)$.}\label{cas:F0}
We choose fibers $C_{0}$
and $F$ of the two rulings on $\mathbb{F}_{0}$ as generators of the
divisor class group of $\mathbb{F}_{0}$. Then a smooth rational curve
$B\subset\mathbb{F}_{0}$ with self-intersection $4$ is linearly
equivalent to either $C_{0}+2F$ or $2C_{0}+F$. It follows that we
may choose bi-homogeneous coordinates $\left[x_{0}:x_{1}\right],\left[y_{0}:y_{1}\right]$
on $\mathbb{F}_{0}$ such that $C_{0}=\left\{ y_{1}=0\right\} $ and
$F=\left\{ x_{1}=0\right\} $ and $B\sim C_{0}+2F$ coincides
with the closure in $\mathbb{F}_{0}$ of the affine conic $\tilde{B}\subset\mathbb{C}^{2}=\mathbb{F}_{0}\setminus\left(C_{0}\cap F\right)={\rm Spec}\left(\mathbb{C}\left[x,y\right]\right)$
defined by the equation $y=x^{2}$. \\

a) Given a point $p= (u, u^2) \in \tilde{B}$
distinct from $p_0 = (0,0)$, the
hyperbola $\tilde{H}$ defined by the equation $\left(x-3u\right)y+3u^{2}\left(x-3u\right)+8u^{3}=0$
intersects $\tilde{B}$ only in $p$, with multiplicity $3$.
Its closure $H$ in $\mathbb{F}_{0}$ is a smooth rational curve linearly
equivalent to $C_{0}+F$. Letting $F_{u}$ be the fiber of the first
ruling over the point $\left[u:1\right]$, it follows that $B$
and $H+F_{u}$ generate a rational pencil $\mathcal{P}$
with $p$ as a unique proper base point (see Fig. \ref{fig:pencils}.a). A minimal resolution of $\mathcal{P}$
is given in Figure \ref{fig:models} (model I).
The restriction of $\mathcal{P}$ to $V\simeq\mathbb{F}_{0}\setminus B_{0}$
is a pencil of affine lines $V\rightarrow\mathbb{A}^{1}$ with a unique
degenerate fiber consisting of the disjoint union of two reduced affine
lines $H\cap V$ and $F_{u}\cap V$.\\

b) In contrast, $B$ is tangent to $C_{0}$ at the point $p_{\infty}=\left[1:0\right],\left[1:0\right]$
so that $B$ and $C_{0}+2F$ generate a rational pencil $\mathcal{P}$
with $p_{\infty}$ as a unique proper base point. A minimal resolution
of $\mathcal{P}$ is described in Figure \ref{fig:models} (model III). In this case, the restriction of $\mathcal{P}$ to $V$ is a pencil of affine
lines $V\rightarrow\mathbb{A}^{1}$ with a unique degenerate fiber
consisting of the disjoint union of a reduced affine line $C_{0}\cap V$
and a nonreduced one $F\cap V$, occurring with multiplicity $2$.
By symmetry, a similar description holds at the point $p_{0}=\left[0:1\right],\left[0:1\right]$. 

\subsubsection{The case $\left(\mathbb{F}_{2},B\right)$.}\label{cas:F2}
We choose a fiber $F$
of the ruling and the exceptional section $C_{0}$ with self-intersection
$-2$ as generators of the divisor class group of $\mathbb{F}_{2}$. The section $B$ is linearly equivalent to $C_{0}+3F$. In
particular $B$ intersects $C_{0}$ transversely in the single point
$C_{0}\cap F$. We identify $\mathbb{F}_{2}\setminus\left(C_{0}\cup F\right)$
to $\mathbb{C}^{2}$ with coordinates $x$ and $y$, the induced ruling
on $\mathbb{C}^{2}$ being given by the first projection. It follows
that $B$ coincides with the closure of an affine cubic $\tilde{B}\subset\mathbb{C}^{2}$
defined by an equation of the form $y=ax^{3}+bx^{2}+cx+d$. Since
every automorphism of $\mathbb{C}^{2}$ of the form $\left(x,y\right)\mapsto\left(\lambda x,\mu y+P\left(x\right)\right)$
where $P$ is a polynomial of degree $d\leq2$, extends to a biregular
automorphism of $\mathbb{F}_{2}$, we may actually assume from the
very beginning that $B$ is the closure in $\mathbb{F}_{2}$
of the curve $\tilde{B} \subset\mathbb{C}^{2}$ defined by the
equation $y=x^{3}$. \\

a) Given a point $p=\left(u,u^{3}\right)\in\tilde{B}$ distinct
from $\left(0,0\right)$, the affine cubic $\tilde{H}$ defined by
the equation $y=3ux^{2}-3u^{2}x+u^{3}$ intersects $\tilde{B}$
only in $p$, with multiplicity $3$. Its closure $H$ in $\mathbb{F}_{2}$
is a smooth rational curve linearly equivalent to $C_{0}+2F$. Let
$F_{u}\subset\mathbb{F}_{2}$ be the fiber of the ruling of $\mathbb{F}_{2}$
over $u$. It follows that $B$ and $H+F_{u}$ generate a rational
pencil $\mathcal{P}$ with
$p$ as a unique proper base point (see Fig. \ref{fig:pencils}.b).  A minimal resolution of $\mathcal{P}$
is given in Figure \ref{fig:models} (model I). The restriction of $\mathcal{P}$
to $V\simeq\mathbb{F}_{2}\setminus B$ is a pencil of affine lines
$V\rightarrow\mathbb{A}^{1}$ with a unique degenerate fiber consisting
of the disjoint union of two reduced affine lines $H\cap V$ and $F_{u}\cap V$.
A similar description holds at the point $p=\left(0,0\right)$ with
the line $\left\{ y=0\right\} $ as the curve $\tilde{H}$. \\ 

\begin{figure}[ht]
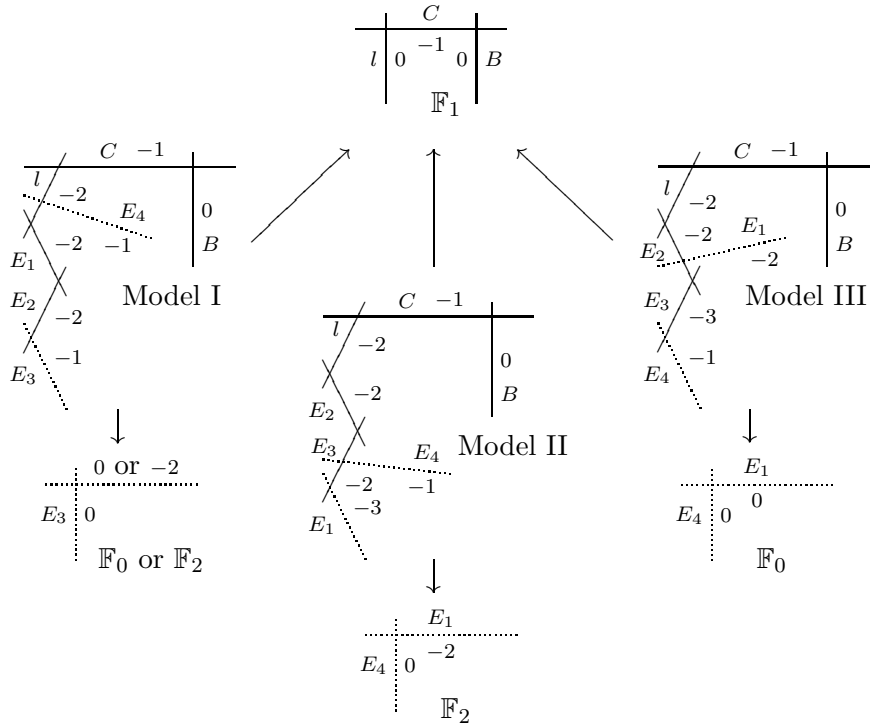

$$\mygraph{
!{<0cm,0cm>;<1cm,0cm>:<0cm,1cm>::}
!{(4.2,4)  }*+{\dessinFun}="F1"
!{(0,0)  }*++{\ModeleI}="MI" 
!{(4.2,-2)  }*++{\ModeleII}="MII"
!{(8.4,0)  }*++{\ModeleIII}="MIII"
"F1"-@{<-}"MI" "F1"-@{<-}"MII" "F1"-@{<-}"MIII" 
}$$
\caption{The three models of fibration (the index of the exceptional divisors $E_i$ corresponds to their order of construction coming from $\F_1$; the dotted curves intersect $V$).}  \label{fig:models}
\end{figure}

b) On the other hand, since $B$ intersects $C_{0}$ transversely
in a unique point $p_{\infty}$, it follows that $B$ and $C_{0}+3F$
generate a rational pencil $\mathcal{P}$
with $p_{\infty}$ as a unique proper base point. A minimal resolution
of $\mathcal{P}$ is described in Figure \ref{fig:models} (model II). As above, the restriction of $\mathcal{P}$ to $V$ is a pencil of affine
lines $V\rightarrow\mathbb{A}^{1}$ with a unique degenerate fiber
consisting of the disjoint union of a reduced affine line $C_{0}\cap V$
and a nonreduced one $F\cap V$, occurring with multiplicity $3$. 

\subsubsection{}\label{cas:tr}
By virtue of the proposition \ref{prop:main}, every automorphism of $V$ can be decomposed
into a sequence of triangular maps between smooth admissible
compactifications of $V$. Hereafter we describe each possible triangular
map that can occur in such a decomposition. \\

a) Triangular map $\left(\mathbb{F}_{2},B\right)\dashrightarrow\left(\mathbb{F}_{2},B'\right)$
with proper base point $B\cap C_{0}$ . 

We let $\hat{S}\rightarrow\mathbb{F}_{2}$ be the minimal resolution
of the base point of the corresponding rational pencil as in \ref{cas:F2}.b), and we let $\hat{\phi}:\hat{S}\dashrightarrow\hat{S}'$ be
a nontrivial fibered modification. The surface $\hat{S}'$ still dominates
$\mathbb{F}_{2}$ via the contraction of the strict transforms of
$C, l, E_2, E_3$. The  strict transform  $B'$ of the last
exceptional divisor produced by the fibered modification is a smooth
rational curve with self-intersection $4$. So $\hat{\phi}$
descends to a triangular map $\phi:\left(\mathbb{F}_{2},B\right)\dashrightarrow\left(\mathbb{F}_{2},B'\right)$
having $B\cap C_{0}$ has a unique proper base point. Note that $\phi^{-1}$
has again $B'\cap C_{0}$ as a unique proper base point. \\

b) Triangular map $\left(\mathbb{F}_{0},B\right)\dashrightarrow\left(\mathbb{F}_{0},B'\right)$
with proper base point at a point $p$ where $B$ is tangent to one
of the members of the two rulings. 

We let $\hat{S}\rightarrow\mathbb{F}_{0}$ be the minimal resolution
of the base point of the corresponding rational pencil as in \ref{cas:F0}.b). Given an arbitrary nontrivial fibered modification $\hat{\phi}:\hat{S}\dashrightarrow\hat{S}'$, the surface $\hat{S'}$ still dominates $\mathbb{F}_{0}$ via the contractions of  $C, l, E_2, E_3$, and the strict transform  $B'$ of the
last exceptional divisor produced by the fibered modification is a
smooth rational curve with self-intersection $4$. The birational
map $\hat{\phi}$ descends to a triangular map $\phi:\left(\mathbb{F}_{0},B\right)\dashrightarrow\left(\mathbb{F}_{0},B'\right)$
with $p$ as a unique proper base point. The proper base point of
$\phi^{-1}$ is again located at a point where $B'$ is tangent to
one of the two rulings. \\

\begin{figure}[ht]
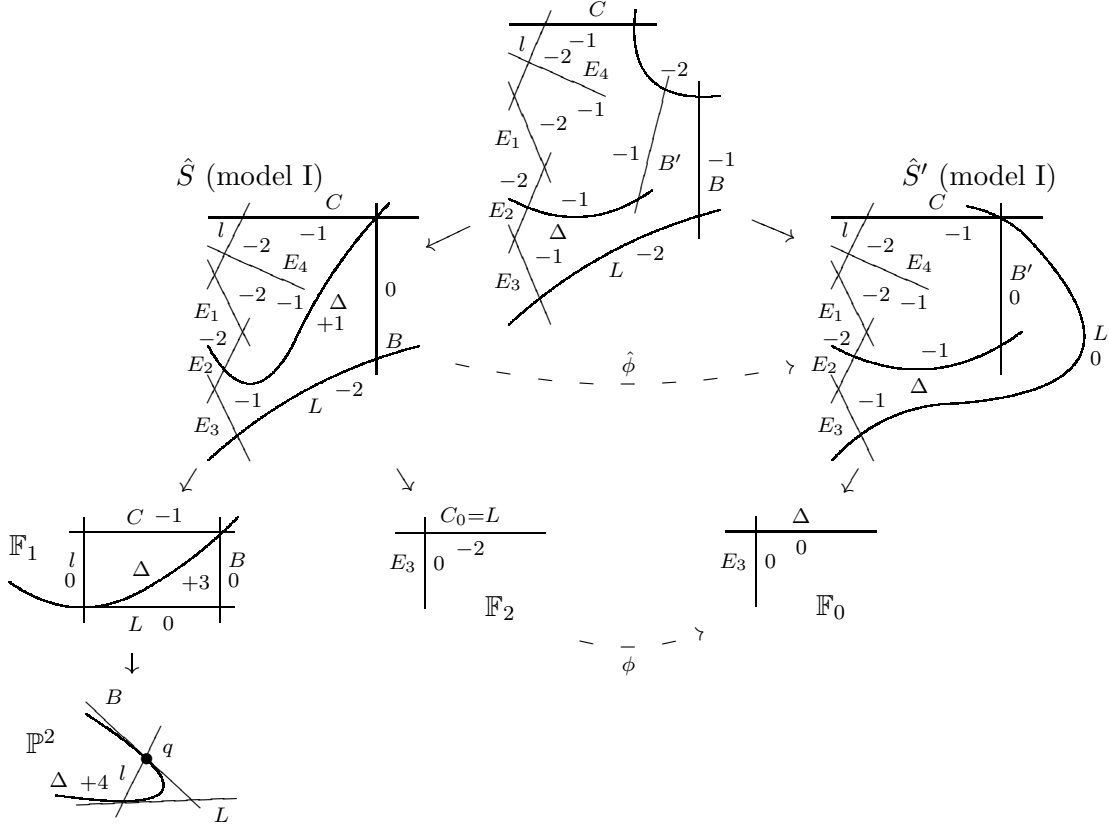

$$\mygraph{
!{<0cm,0cm>;<1.1cm,0cm>:<0cm,1cm>::}
!{(-2,+4)  }*+{\dessinTrDeb}="S"
!{(1.8,+6)  }*+{\dessinTrHaut}="Y"
!{(+6,+4)  }*+{\dessinTrFin}="Shat"
!{(0,.5)  }*+{\dessinTrFdeux}="F2"
!{(-4,.5)  }*+{\dessinTrFun}="F1"
!{(+4,.5)  }*+{\dessinTrFzero}="F0"
!{(-4,-2)  }*+{\dessinTrProj}="P2"
"Y"-@{->}"S" "Y"-@{->}"Shat"
"S"-@{->}"F1" "S"-@{->}"F2"
"Shat"-@{->}"F0" "F1"-@{->}"P2"
"S"-@{-->}@/_1cm/^{\hat{\phi}}"Shat"
"F2"-@{-->}@/_1cm/_{\phi}"F0"
}$$
\caption{Quadratic triangular map $\phi :\F_2 \dashrightarrow \F_0$.}
\label{fig:triangF2F0}
\end{figure}

c) Triangular maps starting from $(\mathbb{F}_{2},B)$
with proper base point contained in $B\setminus C_{0}$. 

We start with the special case of a quadratic triangular map (see Fig. \ref{fig:triangF2F0}).
We let $\hat{S}\rightarrow\mathbb{F}_{0}$ be the minimal resolution
of the base point of the corresponding rational pencil as in \ref{cas:F0}.a)
above. The surface $\hat{S}$ can also be obtained from $\mathbb{P}^{2}$
by a sequence of blow-ups with successive exceptional divisors $C$, $E_{1}$,
$E_{2}$, $E_{3}$ and $E_{4}$  in such a way that the irreducible
curves $l$ and $B$ correspond to the strict transforms of a
pair of lines in $\mathbb{P}^{2}$ intersecting at the center $q$
of the first blow-up. In this setting, the strict transform of $C_{0}\subset\mathbb{F}_{2}$
in $\hat{S}$ coincides with the strict transform of a certain line
$L$ in $\mathbb{P}^{2}$ intersecting $B$ in a point distinct from
$q$. Let $\hat{\phi}:\hat{S}\dashrightarrow\hat{S}'$ be any fibered
modification of degree $2$ and let $B'$ be the strict 
transform in $\hat{S}'$ of the second exceptional divisor produced.
Then one checks that there exists a unique smooth conic $\Delta$
in $\mathbb{P}^{2}$ tangent to $B$ in $q$ and to $L$ at the point
$L\cap l$ such that its strict transform in $\hat{S}'$
is a $\left(-1\right)$-curve which intersects transversely the strict transforms of $E_{2}$ and $B'$ in general points. 
By successively contracting $E_{4},\ldots,E_{1}$, we arrive at a
new projective surface $S'$ in which the strict transform of $B'$
is a smooth rational curve with self-intersection $4$ and such that
the strict transforms of $\Delta$ and $E_{3}$ are smooth rational
curves with self-intersection $0$, intersecting transversely in a
single point. Thus $S'\simeq\mathbb{F}_{0}$ and  $\hat{\phi}:\hat{S}\dashrightarrow\hat{S}'$
descends to triangular map $\phi:\left(\mathbb{F}_{2},B\right)\dashrightarrow\left(\mathbb{F}_{0},B'\right)$.
 Moreover, the proper base point of $\phi^{-1}$ is located at a point
where $B'$ intersects the two rulings transversely. \\

Now we consider triangular maps of degree $3$, that are more representative of the general situation. We will see that in the general case we still obtain maps $\left(\mathbb{F}_{2},B\right)\dashrightarrow\left(\mathbb{F}_{0},B'\right)$, but it will also exists some special triangular maps which end in $\F_2$.  

We fix homogeneous coordinates
$\left[x:y:z\right]$ on $\mathbb{P}^{2}$ in such a way that the
lines $B$, $l$ and $L$ introduced above coincide
with the strict transforms of the lines $\left\{ x=0\right\} $, $\left\{ y=0\right\} $ and $\left\{ z=0\right\} $. Instead of the conic $\Delta$ we consider now the
family of cuspidal cubics $C_{a,b}$ in $\mathbb{P}^{2}$ defined
by equations of the form $y^{3}+axy^{2}+bx^{2}z=0$, where $b\neq0$.
If $a\neq0$ then the strict transform of $C_{a,b}$ in $\hat{S}$
has self-intersection $3$, and intersects $E_{2}$ transversely in a single general point. In contrast,
the strict transform of $C_{0,b}$ has
self-intersection $2$, and intersects $E_{3}$ transversely in a single general point. For any given fibered modification $\hat{\phi}:\hat{S}\dashrightarrow\hat{S}'$
of degree $3$, there exists a unique curve $C_{a_{0},b_{0}}$
among the $C_{a,b}$'s with the property that its strict transform
in $\hat{S}'$ intersects the last exceptional divisor $B'$ produced.
Furthermore, its strict transform is a $-1$-curve if $a_{0}\neq0$
or a $-2$-curve otherwise. In the first (general) case, $\hat{S}'$ dominates
$\mathbb{F}_{0}$ via the contraction of $E_{4},\ldots,E_{1}$, the
two rulings being generated by the strict transforms of $C_{a_{0},b_{0}}$
and $E_{3}$, whereas in the second (special) case $\hat{S}'$ dominates $\mathbb{F}_{2}$
with the strict transform of $C_{0,b_{0}}$ as the exceptional section.
In both cases, the strict transform of $B'$ is a smooth rational
curve with self-intersection $4$, and $\hat{\phi}$ descends to a
triangular map $\phi:\left(\mathbb{F}_{2},B\right)\dashrightarrow\left(\mathbb{F}_{i},B'\right)$.\\ 

More generally, for every fibered modification $\hat{\phi}:\hat{S}\dashrightarrow\hat{S}'$
of degree $d\geq3$, one can check that there exists a cuspidal curve
of degree $d$ in $\mathbb{P}^{2}$ whose strict transform in $\hat{S}'$
intersects $B'$ and has self-intersection $-1$ if it intersects
$E_{2}$ or $-2$ if it intersects $E_{3}$. As in the case of degree 3, $\hat{\phi}$ descends to a triangular map $\phi:\left(\mathbb{F}_{2},B\right)\dashrightarrow\left(S',B'\right)$
where $S'=\mathbb{F}_{0}$ in the first case and $S'=\mathbb{F}_{2}$
in the second one. It also follows from the construction that the
proper base point $q$ of $\phi^{-1}$ is contained in $B'\setminus C_{0}$
if $S'\simeq\mathbb{F}_{2}$. On the other hand, if $S'=\mathbb{F}_{0}$ then
$q$ is a point of $B'$ such that $B'$ is not tangent to the members
of the two rulings of $\mathbb{F}_{0}$ through $q$. \\

d) Triangular maps $\left(\mathbb{F}_{0},B\right)\dashrightarrow\left(\mathbb{F}_{j},B'\right)$,
$j=0,2$ with proper base point at a point where $B$ is not tangent
to the rulings. 

The situation is essentially analogous to the one in c). Namely,
a general triangular map is of the form $\phi:\left(\mathbb{F}_{0},B\right)\dashrightarrow\left(\mathbb{F}_{0},B'\right)$
where the proper base point of $\phi^{-1}$ is of the same type as
the one of $\phi$, and there exists some special triangular maps
$\phi:\left(\mathbb{F}_{0},B\right)\dashrightarrow\left(\mathbb{F}_{2},B'\right)$
which are the inverses of the ones described above. \\

It follows from the definition that the cases a)-d) above exhaust
all possibilities of triangular maps between admissible compactifications
of $V$. One can easily construct some automorphisms of $V$ such that all these types of triangular maps occur in the factorization. Indeed, starting from the the special admissible compactification
$\left(\mathbb{F}_{2},B\right)$ described in \ref{cas:F2}.b), chose  $\phi_{k}:\left(\mathbb{F}_{i_{k}},B_{i_{k}}\right)\dashrightarrow\left(\mathbb{F}_{i_{k+1}},B_{i_{k+1}}\right)$,
$k=0,\ldots,n$ a sequence of triangular maps such that $\left(\mathbb{F}_{i_{0}},B_{i_{0}}\right)=\left(\mathbb{F}_{2},B\right)$, $\left(\mathbb{F}_{i_{n}},B_{i_{n}}\right)=\left(\mathbb{F}_{2},B'\right)$
and such that for every $k=0,\ldots,n-1$, $\phi_{k}$ and  $\phi_{k+1}$ are in general position (that is, the proper base point of
$\phi_{k+1}$ is distinct from the one of $\phi_{k}^{-1}$). Take $f =\phi_{n}\circ\cdots\circ\phi_{0}$. Since
the group of biregular automorphisms of $\F_2$ acts transitively on the set
of sections of the ruling $\mathbb{F}_{2}\rightarrow\mathbb{P}^{1}$,
there exists such an automorphism $\alpha$  inducing
an isomorphism of pairs $\left(\mathbb{F}_{2},B'\right)\stackrel{\sim}{\rightarrow}\left(\mathbb{F}_{2},B\right)$.
Then if we replace $\phi_n$ by $\alpha\circ\phi_n$ (which is still a triangular map), $f:\left(\mathbb{F}_{2},B\right)\dashrightarrow\left(\mathbb{F}_{2},B\right)$
induces an automorphism of $V=\mathbb{F}_{2}\setminus B$
with the property that the triangular maps occurring in its decomposition
are precisely the $\phi_{k}$'s. \\

\subsubsection{}
As a consequence of the description above we get the following

\begin{cor}
 An affine surface $V$ admitting a smooth admissible compactification by a rational curve $B$ with self-intersection $4$ admits precisely three distinct classes of pencils of affine lines $V\rightarrow \mathbb{A}^1$, up to conjugacy by automorphisms.
\end{cor}

\begin{proof}
The closure in a smooth admissible completion $\left(S,B\right)$
of $V$ of such a pencil $\pi:S\rightarrow\mathbb{A}^{1}$ is a rational
pencil $\mathcal{P}$ with a unique proper base point necessarily
supported on $B$. One checks that the proper transform of $B$ in
a minimal resolution $\sigma:\hat{S}\rightarrow S$ of $\mathcal{P}$
is a fiber of the $\mathbb{P}^{1}$-fibration $\bar{\pi}:\hat{S}\rightarrow\mathbb{P}^{1}$
defined by the total transform of $\mathcal{P}$ on $\hat{S}$. Since
$B^{2}=4$ in $S$, it follows that $\sigma^{-1}\left(B\right)\setminus\sigma_{*}^{-1}B$
consists of a chain of four smooth rational curves with self-intersections
$-2$, $-2$, $-2$ and $-1$. It is easilly checked that the rational
pencils of type I, II or III described above (see Fig. \ref{fig:models}) exhaust all possibilities
for the structure of the $\mathbb{P}^{1}$-fibration $\bar{\pi}:\hat{S}\rightarrow\mathbb{P}^{1}$. 
This leads to three distinct classes of such pencils that can already be roughly distinguished from each other by means of the multiplicities of the irreducible components of their degenerate fibers. In view of the description above, it remains to check that up to conjugacy by automorphisms of $V$, there exists only one such pencil with reduced fibers.  By construction, these pencils arise as the restrictions of rational pencils on $(\mathbb{F}_2,B_2)$ and $(\mathbb{F}_0,B_0)$ with a proper base point on $B_2\setminus C_0$ and $B_0\setminus (p_0 \cup p_\infty)$ (see \ref{cas:F0} and \ref{cas:F2}). The fact that the group of biregular automorphisms of $\mathbb{F}_2$ (resp. $\mathbb{F}_0$) preserving the section $B_2$  (resp. $B_0$) acts transitively on $B_2\setminus C_0$ (resp. $B_0\setminus (p_0 \cup p_\infty)$) implies that the pencils obtained from each projective model are conjugated. Now consider $\phi:(\mathbb{F}_0,B_0)\dashrightarrow (\mathbb{F}_2,B)$ a special triangular map (see \ref{cas:tr}.d), and $\alpha$ an automorphism of $\F_2$ that sends $B$ on $B_2$.  The composition $\alpha\circ \phi$ can be interpreted as an automorphism of $V$ sending a given rational pencil with proper point on $B_0\setminus (p_0 \cup p_\infty)$ onto a rational pencil with proper base point on $B_2\setminus C_0$. This shows that all such pencils are conjugated. 
\end{proof}

\subsection{Chain reversions} \label{par:chain}
In this last paragraph we give an example of a factorization where some surfaces with two singularities on the boundary appear. This is related to the existence of chain reversions.

It is known that if a smooth quasi-projective surface $V$  admits a smooth compactification by a chain of rational curves with self-intersections $(-e_1,\cdots,-e_k,-1,0)$, where $e_i\geq -2$ for every $i=1,\ldots ,k$, then it also admits one by a chain of the same length but with reversed self-intersections  $(-e_k,\cdots,-e_1,-1,0)$ (see \textit{e.g.} \cite{G}). Furthermore two such compactifications are always related by a sequence of elementary transformations with centers on the boundary (see \textit{e.g.} \cite{DubML} for explicit log resolutions of these  maps). Starting from such chains one can always produce an admissible compactification of $V$ by first contracting as many successive $-1$-curves as possible to smooth points and then contracting the remaining curves with negative self-intersection to a singular point supported on the strict transform of the initial $0$-curve, which becomes the boundary. 
      
  Here we consider an example which illustrates how these reversions of chains enter the game when one considers a same automorphism of a quasi-projective surface $V$ as a birational transformation between various admissible compactifications.  We let $V$ be the smooth affine surface in $\mathbb{C}^{4}$ defined
by the equations 

\[
\begin{cases}
xz= & y\left(y^{3}-1\right)\\
yu= & z\left(z-1\right)\\
xu= & \left(y^{3}-1\right)\left(z-1\right)\end{cases}\]

One checks that the birational morphism $\pi_{0}:V\rightarrow\mathbb{P}^{2}$,
$\left(x,y,z,u\right)\mapsto\left[x:y:1\right]$ lifts to an open
immersion of $V$ into the smooth projective surface $\overline{V}_{0}$
obtained from $\mathbb{P}^{2}$ with homogeneous coordinates $\left[t_{0}:t_{1}:t_{2}\right]$
by first blowing-up four distinct points on the affine line $L_{0,0}\setminus\left\{ \left[0:1:0\right]\right\} =\left\{ t_{0}=0\right\} \setminus\left\{ \left[0:1:0\right]\right\} $
with exceptional divisors $D_{0,0},D_{0,1},D_{0,2},D_{0,3}$,
and then blowing-up  a point on $D_{0,0}\setminus L_{0,0}$ with exceptional
divisor $D_{0,4}$. The boundary $\overline{V}_{0}\setminus V$ (pictured with plain lines on Fig. \ref{fig:inv0}) consists
of the union of the strict transforms of $L_{0,0}$, $D_{0,0}$ and
the line at infinity $L_{0,\infty}=\left\{ t_{2}=0\right\} $ on $\mathbb{P}^{2}$.
By contracting the strict transforms of $L_{0,0}$ and $D_{0,0}$,
we obtain an admissible compactification $S_{0}$ of $V$, with a unique
singularity of type $A_{5,2}$ supported on its boundary $B_{S_{0}}=L_{0,\infty}$. 
\begin{figure}[ht]
 $$\mygraph{
!{<0cm,0cm>;<1cm,0cm>:<0cm,1cm>::}
!{(0,0)  }*+{\dessinprojzero}="P2" 
!{(4.2,2)  }*+{\dessininversionSzero}="V"
!{(8.4,0)  }*+{\dessinSzero}="S0"
"V"-@{->}"P2" "V"-@{->}"S0" 
}$$
\caption{Sequence of blow-ups and contractions from $\cpd$ to $S_0$.} \label{fig:inv0}
\end{figure}

A second admissible compactification $S_{2}$ of $V$ can be obtained
in a similar way starting from the birational morphism $\pi_{2}:V\rightarrow\mathbb{P}^{2}$,
$\left(x,y,z,u\right)\mapsto\left[u:z:1\right]$. Indeed, one checks
that $\pi_{2}$ lifts to an open immersion of $V$ into the smooth
projective surface $\overline{V}_{2}$ obtained from $\mathbb{P}^{2}$
with homogeneous coordinates $\left[w_{0}:w_{1}:w_{2}\right]$ by
first blowing-up two distinct points on the affine line $L_{2,0}\setminus\left\{ \left[0:1:0\right]\right\} =\left\{ w_{0}=0\right\} \setminus\left\{ \left[0:1:0\right]\right\} $
with exceptional divisors $D_{2,0},D_{2,4}$, and then blowing-up three
distinct points on $D_{2,0}\setminus L_{2,0}$ with exceptional divisors
$D_{2,1}$, $D_{2,2}$ and $D_{2,3}$. The boundary $\overline{V}_{2}\setminus V$
consists of the union of the strict transforms of $L_{2,0}$, $D_{2,0}$
and the line at infinity $L_{2,\infty}=\left\{ w_{2}=0\right\} $
on $\mathbb{P}^{2}$. By contracting the strict transforms of $L_{2,0}$
and $D_{2,0}$, we obtain an admissible compactification $S_{2}$ of
$V$, with a unique singularity of type $A_{3,1}$ supported on its
boundary $B_{S_{2}}=L_{2,\infty}$. 
\begin{figure}[ht]
 $$\mygraph{
!{<0cm,0cm>;<1cm,0cm>:<0cm,1cm>::}
!{(0,0)  }*+{\dessinprojdeux}="P2" 
!{(4.2,2)  }*+{\dessininversionSdeux}="V"
!{(8.4,0)  }*+{\dessinSdeux}="S2"
"V"-@{->}"P2" "V"-@{->}"S2" 
}$$
\caption{Sequence of blow-ups and contractions from $\cpd$ to $S_2$.} \label{fig:inv1}
\end{figure}

The identity morphism ${\rm id}:V\rightarrow V$ induces a birational
map $\sigma:S_{0}\dashrightarrow S_{2}$. The relations \[
\begin{cases}
z & =x^{-1}y\left(y^{3}-1\right)\\
u & =x^{-1}\left(y^{3}-1\right)\left(z-1\right)  =x^{-2}\left(y^{3}-1\right)\left(y\left(y^{3}-1\right)-x\right)\end{cases}\]
in the function field of $V$ imply that there exists a commutative
diagram 

\[\xymatrix{S_0 \ar@{-->}[d] \ar@{-->}[r]^{\sigma} & S_2 \ar@{-->}[d] \\ \mathbb{P}^2 \ar@{-->}[r]^g & \mathbb{P}^2 }\]
where the vertical arrows denote the natural birational morphisms
obtained from the construction of $S_{0}$ and $S_{2}$ and where
$g:\mathbb{P}^{2}\dashrightarrow\mathbb{P}^{2}$ is the birational
map defined by 
\begin{eqnarray*}
 g:\left[t_{0}:t_{1}:t_{2}\right] & \dashrightarrow & \left[w_{0}:w_{1}:w_{2}\right] \\
 &&= \left[\left(t_{1}^{3}-t_{2}^{3}\right) \left(t_{1}\left(t_{1}^{3}-t_{2}^{3}\right)-t_{0}t_{2}^{3}\right):t_{0}t_{1}t_{2}^{2}\left(t_{1}^{3}-t_{2}^{3}\right):t_{0}^{2}t_{2}^{5}\right]
\end{eqnarray*}
 The point $p=\left[1:0:0\right]\in L_{0,\infty}$ is a unique base
point at infinity of $g$ and a resolution of $g$ is obtained by
first blowing-up $p$ with exceptional divisor ${\rm Aux}$, then blowing-up
the point ${\rm Aux}\cap L_{0,\infty}$ with exceptional divisor $E_{1}$
and finally blowing-up ${\rm Aux}\cap E_{1}$ with exceptional divisor
$E_{2}$. This resolution lifts to a log resolution of $\sigma:S_{0}\dashrightarrow S_{2}$
by performing the same sequence of blow-ups over a nonsingular point
of $B_{S_{0}}=L_{0,\infty}$ and then taking a minimal resolution 
of the singularity $A_{5,2}$ of $S_{0}$ by a chain of two rational curves $C_1,C_2$  (see figure \ref{fig:inv2}). 
\begin{figure}[ht]
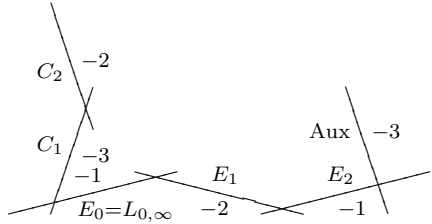

$$\dessinresolutiontranspose$$
\caption{Resolution of $\sigma: S_0 \dashrightarrow S_2$.} \label{fig:inv2}
\end{figure}

It follows that the factorization of $\sigma:S_{0}\dashrightarrow S_{2}$
consists of two links $S_{0}\leftrightarrow S_{1}\leftrightarrow S_{2}$.
Note that the intermediate surface $S_{1}$ has two singularities
of type $A_{3,2}$ and $A_{2,1}$ respectively. By successive blow-ups, one can obtain from  $S_0$ and $S_2$ two distinct compactifications of $V$ 
by chains of type  $(-2,-4,-1,0)$  and  $(-4,-2,-1,0)$  respectively. The birational map $\sigma:S_0\dashrightarrow S_2$ corresponds by construction 
to a reversion of these chains.

Now let $h:V\stackrel{\sim}{\rightarrow}V$ be the unique automorphism
of $V$ lifting the triangular automorphism $\left(u,z\right)\mapsto\left(u,z+u^{2}\right)$
of $\mathbb{C}^{2}$ via the birational morphism $V\rightarrow\mathbb{C}^{2}$,
$\left(x,y,z,u\right)\mapsto\left(u,z\right)$. The birational map
$h_{2}=h:S_{2}\dashrightarrow S_{2}$ admits a resolution by four blow-ups with the first one on the singularity. From this, we get a factorization into
two links $S_{2}\leftrightarrow S_{3}\leftrightarrow S_{2}$. One
checks that the intermediate surface $S_{3}$ is obtained from the
weighted projective plane $\cpd\left(2\right)$ by performing
a sequence of blow-ups and contractions similar to the one used to construct
$S_{2}$ from $\mathbb{P}^{2}$, and $h$ extends to a biregular automorphism
of $S_{3}$. 

One can also consider $h$ as a birational transformation $h_{0}=h:S_{0}\dashrightarrow S_{0}$.
Using again the fact that $h_{0}$ can be interpreted as a lifting
via the natural birational map $S_{0}\dashrightarrow\mathbb{P}^{2}$
of a suitable birational transformation of $\mathbb{P}^{2}$, one
checks that the boundary of a minimal log resolution of $h_{0}$ has
the structure pictured on figure \ref{fig:inv3}. 
\begin{figure}[ht]
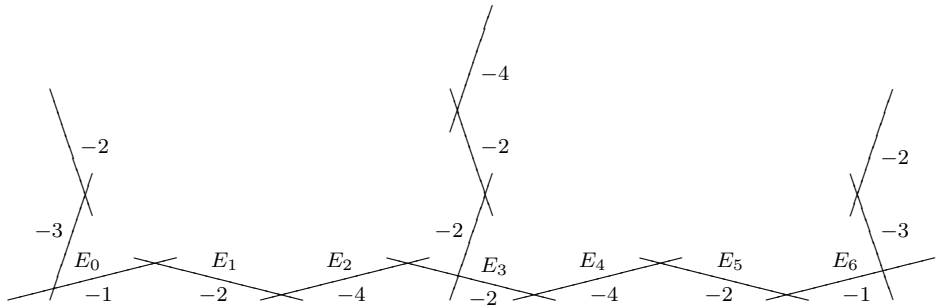

$$\dessinresolutiontransposebis$$
\caption{Resolution of $h_0:S_0 \dashrightarrow S_0$.} \label{fig:inv3}
\end{figure}

\noindent We deduce from this description that the factorization of $h_{0}$
consists of six elementary links 

$$ \underbrace{S_0  \underbrace{\leftrightarrow S_1 \leftrightarrow}_{\sigma}  S_2 \underbrace {\leftrightarrow S_3 \leftrightarrow}_{h_2} S_2 \underbrace{\leftrightarrow S_1 \leftrightarrow}_{\sigma^{-1}} S_0}_{h_0} $$ 
obtained by concatenating the factorizations of $\sigma:S_{0}\dashrightarrow S_{2}$,
$h_{2}:S_{2}\dashrightarrow S_{2}$ and $\sigma^{-1}:S_{2}\dashrightarrow S_{0}$. 

\bibliographystyle{amsplain}
\bibliography{biblio_surface}

\end{document}

%% file: graph_surface.tex
\newcommand{\mygraph}[1]{\xybox{\xygraph{#1}}}
\newcommand{\size}{0.56cm}
\newcommand{\sizemodel}{0.38cm}
\newcommand{\gros}[1]{\mbox{\bf \LARGE \it #1}}

\newcommand{\dessinEzerosing}{
\mygraph{ !{<0cm,0cm>;<\size,0cm>:<0cm,\size>::}
!{(0,0)  }="A"
!{(3,1)  }="B"
!{(0,2)}*{\gros{S}}
"A"-|>(0.2){\bigstar}^{E_0}"B"} 
}

\newcommand{\dessinEn}{
\mygraph{ !{<0cm,0cm>;<\size,0cm>:<0cm,\size>::}
!{(0,1)  }="C"
!{(3,0)  }="D"
!{(3,2)}*{\gros{S'}}
"C"-^{E_n}"D"} 
}

\newcommand{\dessinresolutionbasic}{
\mygraph{
!{<0cm,0cm>;<\size,0cm>:<0cm,\size>::}
!{(0,-2)  }="A" !{(4,-1)  }="B"
!{(3,-1)  }="C" !{(7,-2)  }="D"
!{(6,-2)  }="E" !{(10,-1) }="F"
!{(4.5,-2)}="G" !{(5.5,1)}="H"
!{(5.5,0)}="I" !{(4.5,3)}="J"
"A"-^{C_0}_{-1}"B" "C"-^(0.6){C_3}_(0.6){-2}"D" "E"-^{C_4}_{-1}"F" 
"G"-^{C_2}_{-2}"H" "I"-^{C_1}_{-2}"J"
}}

\newcommand{\dessinresolutionbasicbis}{
\mygraph{
!{<0cm,0cm>;<\size,0cm>:<0cm,\size>::}
!{(0,0)  }="A" !{(4,1)  }="B"
!{(3,1)  }="C" !{(7,0)  }="D"
!{(6,0)  }="E" !{(10,1) }="F"
!{(4.5,0)}="G" !{(5.5,3)}="H"
!{(5.5,2)}="I" !{(4.5,5)}="J"
!{(2,4.5)}="K" !{(5.5,4.5)}="L"
!{(2,4)}="M" !{(5.5,4)}="N"
"A"-^{E_0}_{-1}"B" "C"-^(0.6){E_1}_(0.6){-2}"D" "E"-^{E_2}_{-1}"F" 
"G"-^{\aux_2}_{-2}"H" "I"-^{\aux_1}_{-2}"J"
"K"-@{.}^(0.2){D_{+}}"L" "M"-@{.}_(0.2){D_{-}}"N" 
}}

\newcommand{\dessinresolutiontranspose}{
\mygraph{
!{<0cm,0cm>;<\size,0cm>:<0cm,\size>::}
!{(0,0)  }="A" !{(4,1)  }="B"
!{(3,1)  }="C" !{(7,0)  }="D"
!{(6,0)  }="E" !{(10,1) }="F"
!{(1,0)  }="G" !{(2,3)  }="H"
!{(2,2)  }="I" !{(1,5)  }="J"
!{(9,0)}="K" !{(8,3)}="L"
"A"-_{\quad E_0 = L_{0,\infty}}^{-1}"B" "C"-^{E_1}_{-2}"D" "E"-^{E_2}_{-1}"F" 
"G"-^{C_1}_{-3}"H" "I"-^{C_2}_{-2}"J" "K"-^(0.7){\rm Aux}_(0.6){-3}"L" 
}}

\newcommand{\dessinresolutionKplusB}{
\mygraph{
!{<0cm,0cm>;<\size,0cm>:<0cm,\size>::}
!{(0,0)  }="A" !{(3,1)  }="B"
!{(2,1)  }="C" !{(5,0)  }="D"
!{(4,0)  }="E" !{(7,1) }="F"
!{(6,1)  }="G" !{(9,0)  }="H"
!{(8,0)  }="I" !{(11,1)  }="J"
!{(3,2)}*{\gros{W}}
"A"-^{E_0}|>(0.2){\bigstar}"B" "C"-^{E_1}"D" "E"-@{.}|\bigstar"F" "G"-@{.}|\bigstar"H" "I"-^{E_n}"J" 
}}

\newcommand{\dessinresolutioncomplete}{
\mygraph{
!{<0cm,0cm>;<\size,0cm>:<0cm,\size>::}
!{(0,0)  }="A" !{(3,1)  }="B" !{(2,1)  }="C" !{(5,0)  }="D"
!{(4,0)  }="E" !{(7,1) }="F" !{(6,1)  }="G" !{(9,0)  }="H"
!{(8,0)  }="I" !{(11,1)  }="J"
!{(1,0)  }="K" !{(0,2)  }="L" 
!{(0,1)  }="M" !{(1,3)  }="N" 
!{(6,0)  }="O" !{(5,2)  }="P" 
!{(5,1)  }="Q" !{(6,3)  }="R" 
!{(7,0)  }="S" !{(8,2)  }="T" 
!{(3,2.5)}*{\gros{X}}
"A"-^{E_0}"B" "C"-^{E_1}"D" "E"-@{.}"F" "G"-@{.}"H" "I"-^{E_n}"J"
"K"-"L" "M"-"N" "O"-"P" "Q"-"R" "S"-"T"
}}

\newcommand{\dessinpreuve}{
\mygraph{ !{<0cm,0cm>;<\size,0cm>:<0cm,\size>::}
!{(0,0)  }="A" !{(4,1)  }="B"
!{(3,1)  }="C" !{(7,0)  }="D"
!{(1,0)  }="E" !{(0,3)  }="F"
!{(6,0)  }="G" !{(7,3)  }="H"
!{(7,2)  }="I" !{(6,5)  }="J"
!{(3,3)}*{\gros{Y}}
"A"-_{E_0}"B" "C"-_{E_1}"D"
"E"-@{~}^{\rm Sing}"F" "G"-@{~}^{\rm Aux}"H" "I"-@{~}^{\rm Sing}"J" 
}}

\newcommand{\dessinTa}{
\mygraph{ !{<0cm,0cm>;<\size,0cm>:<0cm,\size>::}
!{(0,-2)  }="A" !{(4,-1)  }="B"
!{(3,-1)  }="C" !{(7,-2)  }="D"
!{(1,-2)  }="E" !{(0,1)  }="F"
!{(6,-2)  }="G" !{(7,1)  }="H"
!{(7,0)  }="I" !{(6,3)  }="J"
"A"-_{E_i}^{-1}"B" "C"-_{E_{i+1}}^{-1}"D"
"G"-@{~}_{\rm Sing}^>(.7){-k,-2,\cdots,-2}"H"  
}}

\newcommand{\dessinTb}{
\mygraph{ !{<0cm,0cm>;<\size,0cm>:<0cm,\size>::}
!{(0,-2)  }="A" !{(4,-1)  }="B"
!{(3,-1)  }="C" !{(7,-2)  }="D"
!{(1,-2)  }="E" !{(0,1)  }="F"
!{(6,-2)  }="G" !{(7,1)  }="H"
!{(7,0)  }="I" !{(6,3)  }="J"
"A"-_{E_i}^{-1}"B" "C"-_{E_{i+1}}^{-1}"D"
"E"-@{~}^{\rm Sing}_>(.7){-k,-2,\cdots,-2}"F"  
}}

\newcommand{\dessinFdeux}{
\mygraph{ !{<0cm,0cm>;<\size,0cm>:<0cm,\size>::}
!{(0,0)  }="A" !{(4,0)  }="B"
!{(3.5,-0.5)  }="C" !{(3.5,3.5)  }="D"
!{(0,3)  }="E" !{(4,3)  }="F"
!{(1,1.5)}*{\gros{$\F_2$}}
"A"-_{E_1}"B" "C"-_{{\rm Aux\,}_2}"D" "E"-_{{\rm Aux\,}_1}"F" 
}}

\newcommand{\dessinVbar}{
\mygraph{ !{<0cm,0cm>;<\size,0cm>:<0cm,\size>::}
!{(0,0)  }="A" !{(4,0)  }="B"
!{(3.5,-0.5)  }="C" !{(3.5,3.5)  }="D"
!{(0,3)  }="E" !{(4,3)  }="F"
!{(1,2)  }="G" !{(0,5)  }="H"
!{(2,2)  }="I" !{(1,5)  }="J"
!{(1,1)}*{\gros{$\bar{V}$}}
"A"-_{E_1}"B" "C"-_{{\rm Aux\,}_2}"D" "E"-_(0.7){{\rm Aux\,}_1}"F"
"G"-@{.}^(0.8){D_{+}}"H" "I"-@{.}_(0.75){D_{-}}"J"
}}

\newcommand{\dessinSun}{
\mygraph{ !{<0cm,0cm>;<\size,0cm>:<0cm,\size>::}
!{(0,0)  }="A" !{(4,0)  }="B"
!{(3.5,-1)  }="C" !{(2,2)  }="D"
!{(2.5,-1)  }="E" !{(4,2)  }="F"
!{(0.5,1)}*{\gros{$S_1$}}
"A"-|>(0.74)\bigstar_{E_1}"B" "C"-@{.}^(0.8){D_{-}}"D" "E"-@{.}_(0.8){D_{+}}"F" 
}}

\newcommand{\dessininversionSzero}{
\mygraph{ !{<0cm,0cm>;<\size,0cm>:<0cm,\size>::}
!{(5.5,1)  }="C" !{(5.5,4.5)  }="D"
!{(1,4)  }="E" !{(6,4)  }="F"
!{(1.5,4.5)}="G" !{(.5,1)}="H"
!{(5,3.5)  }="G1" !{(5,6)}="H1"
!{(3.5,3.5)  }="G2" !{(3.5,6)}="H2"
!{(2,3.5)  }="G3" !{(2,6)}="H3"
!{(0,3)}="I2" !{(2,1)  }="J2"
!{(4,1)}*{\gros{$\bar{V_0}$}}
"C"-^(0.4){+1}_(0.4){L_{0,\infty}}"D" "E"-^(0.62){-3}_(0.62){L_{0,0}}"F" 
"G"-^(0.4){-2}_(0.3){D_{0,0}}"H" 
"G1"-@{.}_(0.9){D_{0,3}}"H1" "G2"-@{.}_(0.9){D_{0,2}}"H2" "G3"-@{.}_(0.9){D_{0,1}}"H3" "I2"-@{.}^(0.8){D_{0,4}}"J2"
}}

\newcommand{\dessininversionSdeux}{
\mygraph{ !{<0cm,0cm>;<\size,0cm>:<0cm,\size>::}
!{(5.5,1)  }="C" !{(5.5,4.5)  }="D"
!{(1,4)  }="E" !{(6,4)  }="F"
!{(1.5,4.7)  }="G" !{(.5,1)  }="H"
!{(3.5,3.5)  }="G2" !{(3.5,6)}="H2"
!{(0,4)}="I1" !{(3,3)  }="J1"
!{(0,3)}="I2" !{(2.5,1.5)  }="J2"
!{(0,2)}="I3" !{(2,0)  }="J3"
!{(4,1)}*{\gros{$\bar{V_2}$}}
"C"-^(0.4){+1}_(0.4){L_{2,\infty}}"D" "E"-^(0.7){-1}_(0.7){L_{2,0}}"F" "G"-^(0){-4}_(0.05){D_{2,0}}"H" 
"G2"-@{.}_(0.9){D_{2,4}}"H2"  
"I1"-@{.}_(0.9){D_{2,1}}"J1" "I2"-@{.}_(0.9){D_{2,2}}"J2" "I3"-@{.}_(0.9){D_{2,3}}"J3"
}}

\newcommand{\dessinprojzero}{
\mygraph{ !{<0cm,0cm>;<\size,0cm>:<0cm,\size>::}
!{(3,0)}="A" !{(3,3.5)}="B"
!{(0,3)}="C" !{(3.5,3)}="D"
!{(0.5,4.5)}*{\gros{$\cpd$}}
"A"-^(0.3){+1}_(0.3){L_{0,\infty}}"B" "C"-^{+1}_{L_{0,0}}"D"
}}

\newcommand{\dessinprojdeux}{
\mygraph{ !{<0cm,0cm>;<\size,0cm>:<0cm,\size>::}
!{(3,0)}="A" !{(3,3.5)}="B"
!{(0,3)}="C" !{(3.5,3)}="D"
!{(0.5,4.5)}*{\gros{$\cpd$}}
"A"-^(0.3){+1}_(0.3){L_{2,\infty}}"B" "C"-^{+1}_{L_{2,0}}"D"
}}

\newcommand{\dessinSzero}{
\mygraph{ !{<0cm,0cm>;<\size,0cm>:<0cm,\size>::}
!{(3,0)}="A" !{(3,3.5)}="B"
!{(2,0)}="C" !{(3.5,3)}="D"
!{(1,1)}="E" !{(4,2.5)}="F"
!{(1,2)}="G" !{(4,2)}="H"
!{(1,4)}="I" !{(4,1)}="J"
!{(4,4.5)}*{\gros{$S_0$}}
"A"-|(0.56)\bigstar_(0.1){L_{0,\infty}}"B" 
"C"-@{.}^(0.1){D_{0,1}}"D" "E"-@{.}^(0.1){D_{0,2}}"F" 
"G"-@{.}^(0.1){D_{0,3}}"H" "I"-@{.}^(0.1){D_{0,4}}"J"
}}

\newcommand{\dessinSdeux}{
\mygraph{ !{<0cm,0cm>;<\size,0cm>:<0cm,\size>::}
!{(3,0)}="A" !{(3,3.5)}="B"
!{(2,0)}="C" !{(3.5,3)}="D"
!{(1,1)}="E" !{(4,2.5)}="F"
!{(1,2)}="G" !{(4,2)}="H"
!{(1,4)}="I" !{(4,1)}="J"
!{(4,4.5)}*{\gros{$S_2$}}
"A"-|(0.56)\bigstar_(0.1){L_{2,\infty}}"B" 
"C"-@{.}^(0.1){D_{2,1}}"D" "E"-@{.}^(0.1){D_{2,2}}"F" 
"G"-@{.}^(0.1){D_{2,3}}"H" "I"-@{.}^(0.1){D_{2,4}}"J"
}}

\newcommand{\dessinresolutiontransposebis}{
\mygraph{
!{<0cm,0cm>;<\size,0cm>:<0cm,\size>::}
!{(0,0)  }="A0" !{(4,1)  }="B0"
!{(3,1)  }="A1" !{(7,0)  }="B1"
!{(6,0)  }="A2" !{(10,1) }="B2"
!{(9,1)  }="A3" !{(13,0)  }="B3"
!{(12,0)  }="A4" !{(16,1)  }="B4"
!{(15,1)  }="A5" !{(19,0) }="B5"
!{(18,0)  }="A6" !{(22,1) }="B6"
!{(1,0)  }="G1" !{(2,3)  }="G2" !{(2,2)  }="G3" !{(1,5)  }="G4"
!{(10.5,0)}="H1" !{(11.5,3)}="H2" 
!{(11.5,2)}="H3" !{(10.5,5)}="H4"  !{(10.5,4)}="H5" !{(11.5,7)}="H6" 
!{(21,0)  }="I1" !{(20,3)  }="I2" !{(20,2)  }="I3" !{(21,5)  }="I4"
"A0"-^{E_0}_{-1}"B0" "A1"-^{E_1}_{-2}"B1" "A2"-^{E_2}_{-4}"B2" 
"A3"-^(0.6){E_3}_(0.6){-2}"B3" "A4"-^{E_4}_{-4}"B4" "A5"-^{E_5}_{-2}"B5" 
"A6"-^{E_6}_{-1}"B6" 
"G1"-^{-3}"G2" "G3"-_{-2}"G4" 
"H1"-^{-2}"H2" "H3"-_{-2}"H4" "H5"-_{-4}"H6" 
"I1"-_{-3}"I2" "I3"-_{-2}"I4" 
}}

\newcommand{\dessinModeleI}{
\mygraph{
!{<0cm,0cm>;<\size,0cm>:<0cm,\sizemodel>::}
!{(2,-2) }="A1" !{(1,1)  }="A2" 
!{(1,0)  }="B1" !{(2,3)  }="B2" 
!{(2,2)  }="C1" !{(1,5)  }="C2"
!{(1,4)  }="D1" !{(2,7)  }="D2"
!{(1,6.5)  }="E1" !{(6,6.5)  }="E2"
!{(5,7)  }="F1" !{(5,3)  }="F2"
!{(1,5.5)  }="S1" !{(4,4)  }="S2"
!{(4.5,2)}*{\mbox{Model I}}
"A1"-@{.}^{E_3}_{-1}"A2" "B1"-^{E_2}_{-2}"B2" "C1"-^{E_1}_{-2}"C2" "D1"-^>(.6)l_>(.6){-2}"D2" 
"E1"-^>(.4){C}^>(.6){-1}"E2" "F1"-^{0}^>(.8)B"F2"
"S1"-@{.}^>(0.8){E_4}_>(0.8){-1}"S2"
}}

\newcommand{\dessinModeleII}{
\mygraph{
!{<0cm,0cm>;<\size,0cm>:<0cm,\sizemodel>::}
!{(2,-2) }="A1" !{(1,1)  }="A2" 
!{(1,0)  }="B1" !{(2,3)  }="B2" 
!{(2,2)  }="C1" !{(1,5)  }="C2"
!{(1,4)  }="D1" !{(2,7)  }="D2"
!{(1,6.5)  }="E1" !{(6,6.5)  }="E2"
!{(5,7)  }="F1" !{(5,3)  }="F2"
!{(1,1.5)  }="S1" !{(4,1)  }="S2"
!{(5.5,2)}*{\mbox{Model II}}
"A1"-@{.}^{E_1}_{-3}"A2" "B1"-^>(.55){E_3}_>(.3){-2}"B2" "C1"-^{E_2}_{-2}"C2" "D1"-^>(.6)l_>(.6){-2}"D2" 
"E1"-^>(.4){C}^>(.6){-1}"E2" "F1"-^{0}^>(.8)B"F2"
"S1"-@{.}^>(0.8){E_4}_>(0.8){-1}"S2"
}}

\newcommand{\dessinModeleIII}{
\mygraph{
!{<0cm,0cm>;<\size,0cm>:<0cm,\sizemodel>::}
!{(2,-2) }="A1" !{(1,1)  }="A2" 
!{(1,0)  }="B1" !{(2,3)  }="B2" 
!{(2,2)  }="C1" !{(1,5)  }="C2"
!{(1,4)  }="D1" !{(2,7)  }="D2"
!{(1,6.5)  }="E1" !{(6,6.5)  }="E2"
!{(5,7)  }="F1" !{(5,3)  }="F2"
!{(1,3)  }="S1" !{(4,4)  }="S2"
!{(4.5,2)}*{\mbox{Model III}}
"A1"-@{.}^{E_4}_{-1}"A2" "B1"-^{E_3}_{-3}"B2" "C1"-^>(.6){E_2}_>(.6){-2}"C2" "D1"-^l_{-2}"D2" 
"E1"-^>(.4){C}^>(.6){-1}"E2" "F1"-^{0}^>(.8)B"F2"
"S1"-@{.}^>(0.8){E_1}_>(0.8){-2}"S2"
}}

\newcommand{\dessinFun}{
\mygraph{
!{<0cm,0cm>;<0.4cm,0cm>:<0cm,0.4cm>::}
!{(2,4)  }="D1" !{(2,1)  }="D2"
!{(1,3.5)}="E1" !{(6,3.5)  }="E2"
!{(5,4)  }="F1" !{(5,1)  }="F2"
!{(4,1)}*{\F_1}
"D1"-_l^{0}"D2" "E1"-^C_{-1}"E2" "F1"-_0^{B}"F2"
}}

\newcommand{\dessinFdeuxbis}{
\mygraph{
!{<0cm,0cm>;<0.4cm,0cm>:<0cm,0.4cm>::}
!{(1,3.5)}="E1" !{(6,3.5)  }="E2"
!{(2,4)  }="F1" !{(2,1)  }="F2"
!{(4,1)}*{\F_2}
"E1"-@{.}_{-2}^{E_1}"E2" "F1"-@{.}^{0}_{E_4}"F2"
}}

\newcommand{\dessinFzero}{
\mygraph{
!{<0cm,0cm>;<0.4cm,0cm>:<0cm,0.4cm>::}
!{(1,3.5)}="E1" !{(6,3.5)  }="E2"
!{(2,4)  }="F1" !{(2,1)  }="F2"
!{(4,1)}*{\F_0}
"E1"-@{.}_{0}^{E_1}"E2" "F1"-@{.}^{0}_{E_4}"F2"
}}

\newcommand{\dessinFzerobis}{
\mygraph{
!{<0cm,0cm>;<0.4cm,0cm>:<0cm,0.4cm>::}
!{(1,3.5)}="E1" !{(6,3.5)  }="E2"
!{(2,4)  }="F1" !{(2,1)  }="F2"
!{(4.5,1)}*{\F_0 \mbox{ or } \F_2}
"F1"-@{.}^{0}_{E_3}"F2" "E1"-@{.}^>(.6){0 \mbox{ or }-2}"E2"
}}

\newcommand{\ModeleI}{
\mygraph{
!{<0cm,0cm>;<\size,0cm>:<0cm,\size>::}
!{(0,0)  }*+{\dessinFzerobis}="F0"
!{(0,6)  }*{\dessinModeleI}="MI"
"MI"-@{->}"F0" 
}}

\newcommand{\ModeleII}{
\mygraph{
!{<0cm,0cm>;<\size,0cm>:<0cm,\size>::}
!{(0,0)  }*+{\dessinFdeuxbis}="F2"
!{(0,6)  }*{\dessinModeleII}="MII"
"MII"-@{->}"F2" 
}}

\newcommand{\ModeleIII}{
\mygraph{
!{<0cm,0cm>;<\size,0cm>:<0cm,\size>::}
!{(0,0)  }*+{\dessinFzero}="F0"
!{(0,6)  }*{\dessinModeleIII}="MIII"
"MIII"-@{->}"F0" 
}}

\newcommand{\courb}{1.13cm}

\newcommand{\dessinDansFzero}{
\mygraph{
!{<0cm,0cm>;<\size,0cm>:<0cm,\size>::}
!{(0,0)}*{\bullet}="p" !{(0,-2)}="u"
!{(0,-3)}="Fu1" !{(0,3)}="Fu2"
!{(-4,0)}="deb" !{(+4,0)}="fin" 
!{(-1.5,-3)}="L1" !{(1.5,3)}="L2"
!{(-3,2)}="H1" !{(3,2)}="H2"
!{(-3,-2)}="B1" !{(3,-2)}="B2"
!{(2,-3)}="D1" !{(2,3)}="D2"
!{(-2,-3)}="G1" !{(-2,3)}="G2"
!{(-3,3)}*{\F_0}
"H1"-^>(.4){C_0}@{.}"H2" "B1"-@{.}_>(0.25){x}"B2"
"D1"-_>(.3){F}@{.}"D2" "G1"-@{.}^>(0.25){y}"G2" 
"L1"-|<(-.05){H \sim C_0+F}@{.}"L2"
"deb"-^<{y=x^2}@/_\courb/"p" "p"-^>(0.85){B \sim C_0+2F}@/^\courb/"fin"
"Fu1"-_>(.3){F_u}_{p}@{.}"Fu2"
}}

\newcommand{\dessinDansFdeux}{
\mygraph{
!{<0cm,0cm>;<\size,0cm>:<0cm,\size>::}
!{(0,-1.1)}*{\bullet}="p" !{(0,-2)}="u"
!{(-5,-4)}="deb" !{(3,4)}="fin" 
!{(-2,-2)}="ori" !{(0,-1)}="mil" 
!{(0,-3)}="Fu1" !{(0,3)}="Fu2"
!{(-1,-3)}="L1" !{(4,0)}="L2"
!{(-3,2)}="H1" !{(3,2)}="H2"
!{(-3,-2)}="B1" !{(3,-2)}="B2"
!{(2,-3)}="D1" !{(2,3)}="D2"
!{(-2,-3)}="G1" !{(-2,3)}="G2"
!{(-3,3)}*{\F_2}
"H1"-^>(.4){C_0}_>(.3){-2}@{.}"H2" "B1"-@{.}_>(0.25){x}"B2"
"D1"-_>(.4){F}@{.}"D2" "G1"-@{.}^>(0.25){y}"G2" 
"deb"-_>(.2){y=x^3}@/^.35cm/"ori" "ori"-@/_.2cm/"mil" 
"mil"-_>(0.85){B \sim C_0+3F}@/_.1cm/"fin"
"L1"-^>(0.95){H \sim C_0+2F}@/^.76cm/@{.}"L2"
"Fu1"-^>(.4){F_u}_>(.3){p}@{.}"Fu2"
}}

\newcommand{\dessinTrFzero}{
\mygraph{
!{<0cm,0cm>;<0.4cm,0cm>:<0cm,0.4cm>::}
!{(1,3.5)}="E1" !{(6,3.5)  }="E2"
!{(2,4)  }="F1" !{(2,1)  }="F2"
!{(4.5,1)}*{\F_0}
"F1"-^{0}_{E_3}"F2" "E1"-_0^{\Delta}"E2"
}}

\newcommand{\dessinTrFun}{
\mygraph{
!{<0cm,0cm>;<0.4cm,0cm>:<0cm,0.4cm>::}
!{(.5,3.5)}="H1" !{(6,3.5)  }="H2"
!{(1,4)  }="G1" !{(1,.5)  }="G2"
!{(.5,1)}="B1" !{(6,1)  }="B2"
!{(5.5,4)  }="D1" !{(5.5,.5)  }="D2"
!{(-1.5,1.85)  }="delD" !{(3.25,1.75)  }="delM" !{(6.1,4)  }="delF"
!{(-1,3)}*{\F_1}
"G1"-_>(.4){l}_>(.6){0}"G2" "H1"-^>(.6){-1}^>(.4){C}"H2"
"B1"-_>(.6){0}_>(.4){L}"B2" "D1"-^>(.4){B}^>(.6){0}"D2"
"delD"-@/_.32cm/"delM" "delM"-^<{\Delta}_>(.3){+3}@/_.05cm/"delF"
}}

\newcommand{\dessinTrFdeux}{
\mygraph{
!{<0cm,0cm>;<0.4cm,0cm>:<0cm,0.4cm>::}
!{(1,3.5)}="E1" !{(6,3.5)  }="E2"
!{(2,4)  }="F1" !{(2,1)  }="F2"
!{(4.5,1)}*{\F_2}
"F1"-^{0}_{E_3}"F2" "E1"-^{C_0= L}_{-2}"E2"
}}

\newcommand{\dessinTrProj}{
\mygraph{
!{<0cm,0cm>;<0.4cm,0cm>:<0cm,0.4cm>::}
!{(2.5,3.5)  }="G1" !{(1,.5)  }="G2"
!{(-.3,0.9)}="B1" !{(5,1.1)  }="B2"
!{(0,4.3)  }="D1" !{(3.8,.8)  }="D2"
!{(-1,1.2)  }="delD" !{(0,3.9)  }="delF"
!{(2,2.4)}*{\bullet}
!{(-1.5,3)}*{\mathbb{P}^2}
"G1"-_>(.6){l}"G2" 
"B1"-_>(.9){L}"B2" "D1"-^>(.1){B}^>(.6){q}"D2"
"delD"-@/_1.3cm/^>(.1){+4}^<{\Delta}"delF" 
}}

\newcommand{\dessinTrHaut}{
\mygraph{
!{<0cm,0cm>;<\size,0cm>:<0cm,\sizemodel>::}
!{(2,-2) }="A1" !{(1,1.5)  }="A2" 
!{(1,0.5)  }="B1" !{(2,4)  }="B2" 
!{(2,3)  }="C1" !{(1,6.5)  }="C2"
!{(1,5.5)  }="D1" !{(2,9)  }="D2"
!{(1,8.5)  }="E1" !{(4.5,8.5)  }="E2"
!{(5.5,6.5)  }="F1" !{(5.5,1)  }="F2"
!{(4.8,6.7)  }="B'H" !{(4,1.9)  }="B'B"
!{(1,7.5)  }="S1" !{(3.3,6)  }="S2"
!{(4,9)  }="ExH" !{(6,6)  }="ExB"
!{(1,2.4)  }="delD" !{(4.4,2.7)  }="delF" 
!{(1,-2)  }="BD" !{(6,2)  }="BF" 
"A1"-^{E_3}_>(.6){-1}"A2" "B1"-^>(.35){E_2}^>(.65){-2}"B2" "C1"-^{E_1}_{-2}"C2" "D1"-^>(.6)l_>(.6){-2}"D2" 
"B'H"-^>(.6){B'}_>(.6){-1}"B'B" 
"E1"-_{-1}^>(.6)C"E2" "F1"-^{-1}^>(.65)B"F2"
"S1"-^>(0.8){E_4}_>(0.95){-1}"S2" 
"ExH"-^>(.6){-2}@/_0.5cm/"ExB" 
"delD"-_>(.35){\Delta}^>(.45){-1}@/_0.3cm/"delF"
"BD"-@/^0.2cm/_>(.5){L}_>(.65){-2}"BF"
}}

\newcommand{\dessinTrDeb}{
\mygraph{
!{<0cm,0cm>;<\size,0cm>:<0cm,\sizemodel>::}
!{(2,-2) }="A1" !{(1,1)  }="A2" 
!{(1,0)  }="B1" !{(2,3)  }="B2" 
!{(2,2)  }="C1" !{(1,5)  }="C2"
!{(1,4)  }="D1" !{(2,7)  }="D2"
!{(1,6.5)  }="E1" !{(6,6.5)  }="E2"
!{(5,7)  }="F1" !{(5,1)  }="F2"
!{(1,5.5)  }="S1" !{(3.3,4)  }="S2"
!{(1,2)  }="delD" !{(3,2)  }="delM" !{(5.3,7)  }="delF" 
!{(1,-2)  }="BD" !{(6,2)  }="BF" 
!{(2,8)}*{\hat{S}\mbox{ (model I)}}
"A1"-^{E_3}_>(.6){-1}"A2" "B1"-^>(.35){E_2}^>(.65){-2}"B2" "C1"-^{E_1}_{-2}"C2" "D1"-^>(.6)l_>(.6){-2}"D2" 
"E1"-_{-1}^>(.6)C"E2" "F1"-^{0}^>(.8)B"F2"
"S1"-^>(0.8){E_4}_>(0.95){-1}"S2" 
"delD"-@/_.5cm/"delM" "delM"-_>(.35){\Delta}_>(.2){+1}@/^0.1cm/"delF"
"BD"-@/^0.2cm/_>(.5){L}_>(.65){-2}"BF"
}}

\newcommand{\dessinTrFin}{
\mygraph{
!{<0cm,0cm>;<\size,0cm>:<0cm,\sizemodel>::}
!{(2,-2) }="A1" !{(1,1)  }="A2" 
!{(1,0)  }="B1" !{(2,3)  }="B2" 
!{(2,2)  }="C1" !{(1,5)  }="C2"
!{(1,4)  }="D1" !{(2,7)  }="D2"
!{(1,6.5)  }="E1" !{(6,6.5)  }="E2"
!{(5,7)  }="F1" !{(5,1)  }="F2"
!{(1,5.5)  }="S1" !{(3.3,4)  }="S2"
!{(1,2)  }="delD" !{(5.5,2.5)  }="delF" 
!{(1,-2)  }="BD" !{(4,0)  }="BM1" 
!{(4.2,6.9)  }="BF" !{(5.5,6)  }="BM2" 
!{(4.5,8)}*{\hat{S}'\mbox{ (model I)}}
"A1"-^{E_3}_>(.6){-1}"A2" "B1"-^>(.35){E_2}^>(.65){-2}"B2" "C1"-^{E_1}_{-2}"C2" "D1"-^>(.6)l_>(.6){-2}"D2" 
"E1"-_>(.6){-1}^C"E2" "F1"-^>(.55){0}^>(.4){B'}"F2"
"S1"-^>(0.8){E_4}_>(0.95){-1}"S2" 
"delD"-@/_.4cm/_>(.45){\Delta}^>(.55){-1}"delF"
"BD"-@/^0.2cm/"BM1" "BM1"-@/_1.3cm/_>(.5){0}_>(.6){L}"BM2" 
"BM2"-@/_0.05cm/"BF"
}}